\documentclass[conference,a4paper]{IEEEtran}

\usepackage[usenames,dvipsnames]{pstricks}
\usepackage{epsfig}
\usepackage{pst-grad} 
\usepackage{pst-plot} 
\usepackage{color}
\usepackage{graphicx}
\usepackage{subcaption}

\usepackage{pgf}
\usepackage{tikz}
\usetikzlibrary{shapes,external}
\usepackage{graphicx}
\usepackage[usenames,dvipsnames]{pstricks}
\usepackage{epsfig}
\usepackage{pst-grad} 
\usepackage{pst-plot} 
\usepackage[space]{grffile} 
\usepackage{etoolbox} 
\makeatletter 
\patchcmd\Gread@eps{\@inputcheck#1 }{\@inputcheck"#1"\relax}{}{}
\makeatother

\usepackage{pst-all} 
\usepackage{mathbbol}
\usepackage[utf8]{inputenc} 
\usepackage[T1]{fontenc}
\usepackage{url}
\usepackage{ifthen}
\usepackage[cmex10]{amsmath} 
                          
\newcommand{\X}{\mathcal{X }}

\newcommand{\ML}{\textsf{ML} }
\newcommand{\CI}{\textsf{CI }}
\newcommand{\N}{\mathrm{N}}
\newcommand{\Ll}{\mathrm{L}}
\newcommand{\f}{\mathrm{f}}
\newcommand{\D}{\mathcal{D}}
\usepackage{amsthm}
\newtheorem{thm}{Theorem}
\newtheorem{lem}[thm]{Lemma}

\newtheorem{cor}[thm]{Corollary}
\newtheorem{dif}[thm]{Definition}

\begin{document}
\title{Information Theoretic Bounds on Optimal Worst-case Error in Binary Mixture Identification} 


\author{%
  \IEEEauthorblockN{Khashayar Gatmiry and Seyed Abolfazl Motahari}
  \IEEEauthorblockA{Computer Engineering Department\\
                    Sharif University of Technology, Tehran, Iran\\
                    Email: kgatmiry@ce.sharif.edu, motahari@sharif.edu}
}


\maketitle

\begin{abstract}
Identification of latent binary sequences from a pool of noisy observations has a wide range of applications in both statistical learning and population genetics. Each observed sequence is the result of passing one of the latent mother-sequences through a binary symmetric channel, which makes this configuration analogous to a special case of Bernoulli Mixture Models. This paper aims to attain an asymptotically tight upper-bound on the error of Maximum Likelihood mixture identification in such problems. The obtained results demonstrate fundamental guarantees on the inference accuracy of the optimal estimator. To this end, we set out to find the closest pair of discrete distributions with respect to the Chernoff Information measure.  We provide a novel technique to lower bound the Chernoff Information in an efficient way. We also show that a drastic phase transition occurs at noise level 0.25. Our findings reveal that the identification problem becomes much harder as the noise probability exceeds this threshold. 
\end{abstract}


\section{Introduction}

Identification of latent parameters of Bernoulli Mixture Models (\textsf{BMM}) has many applications in Statistical Learning and Bioinformatics. In Bioinformatics, next-generation sequencing technologies provide noisy observations of a vast number of sequences and the target is to find the unobserved and latent source sequences \cite{motahari2013information,motahari2013optimal}. In this paper, we aim at obtaining information theoretic bounds on reliable identification of such sources.

Learning parameters of a \textsf{BMM} is not always feasible as there exist district source parameters providing the same output model. The problem is known as identification of \textsf{BMM}s that has been addressed in several papers \cite{allman2009identifiability,gyllenberg1994non,najafi2017reliable,carreira2000practical}.

In this paper, bounding the  Maximum Likelihood (\textsf{ML}) estimator which yields the optimum decision making, we obtain several interesting results regarding identifiability of \textsf{BMM}s in our worse-case analysis.  First, we obtain asymptotically tight upper-bounds on the error of \ML estimator.  Second, we provide a systematic procedure which can be used to efficiently bound the Chernoff Information (\textsf{CI}) measure. Even though \CI is not analytically computable, the interesting structure of the distribution space leads to analytical closed forms for the minimum \CI distance in special regimes of the parameters, and near-optimal bounds for the other cases.  Finally, we demonstrate an astonishing phase transition in our worst-case analysis:  the closest pairs of sources that attain our bounds asymptotically have different characteristics depending on the noise rate. The threshold for the noise level is derived analytically which is  \%25. In Fig. \ref{fig:image1}, the upper bounds on \ML are drawn for different values of noise levels. As it can be seen, the two bounds cross at 0.25 revealing different worst-case scenarios for the two regimes.

 \begin{figure}[t]
   \centering
   \includegraphics[width=0.47\textwidth]{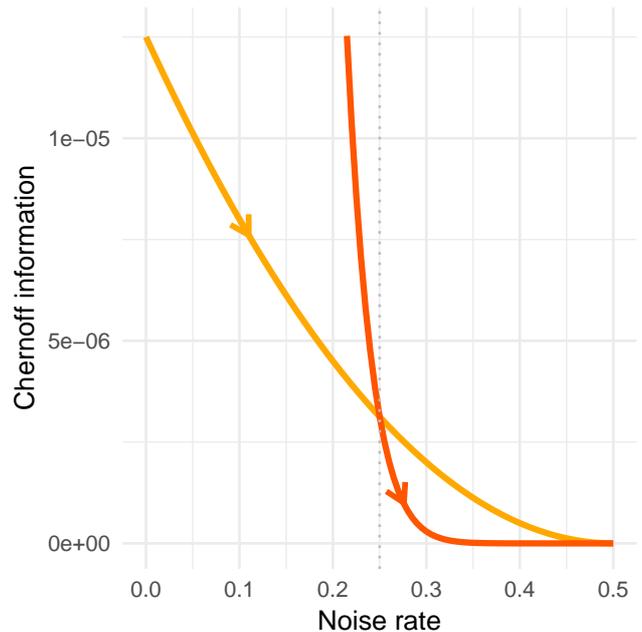}
   \caption{Phase transition: the two upper bounds cross at 0.25.}
   \label{fig:image1}
 \end{figure}

Our findings can also be useful for Information Geometry research, as there have been various attempts to analyze the \CI in parametric distribution spaces with wide applications ranging from signal processing to machine learning \cite{nielsen2013information}.

The paper is organized as follows. In Section \ref{seq:pf}, the problem formulation is presented. In Section \ref{seq:main}, our main results are presented. We provide the proof ideas of our main results in Section \ref{seq:proof}. Finally, in Section \ref{seq:con}, we conclude the paper. 

\section{Problem formulation}\label{seq:pf}

We consider a source having $k$ symbols where the frequency of the $i$th symbol is denoted by $\mathrm{p}_{i}$. Furthermore, we assume symbols are distinct binary vectors of length $\mathrm{L}$. The $i$th symbol is denoted by $\mathcal{Z}_{i}$. The source is observed through a symmetric memory less noisy channel, where we have access to $m$ i.i.d observations of the source from the channel. In particular, at time $1 \leq j \leq m$, the source outputs the symbol $\mathcal{Y}_{j}$ based on the frequencies of the symbols and we observe $\mathcal{M}_{j}$ where $\mathcal{P}(\mathcal{M}_{j} \mid \mathcal{Y}_{j}) = \prod_{l=1}^L p(\mathcal{M}_{j}(l) \mid \mathcal{Y}_{j}(l))$, and $p$ is defined as 
\begin{align}
p(x \mid y) = \begin{cases}1- \mathrm{f} & x=y\\\mathrm{f} & x \neq y\end{cases}. \nonumber
\end{align}
The flip probability $\mathrm{f}$ is known. We are interested in learning the source symbols and their frequencies. 

Given a fixed number N, we assume all the frequencies are integer multiples of $\frac{1}{\mathrm{N}}$. In this way, the frequency of $\mathcal{Z}_{i}$ can be expressed as $\frac{\alpha_{i}}{\mathrm{N}}$ where $\alpha_{i} \in \mathbb{N}$. Regarding this assumption, the source can be equivalently represented as an $\mathrm{N} \times \mathrm{L}$ matrix $\mathcal{X}^*$, where each row corresponds to one source symbol, and symbol $\mathcal{Z}_{i}$ is replicated $\alpha_{i}$ times. The distribution of the source can be expressed as $P_{\mathcal{X}^{*}}$. 

Let $\mathcal{M} = \{ \mathcal{M}_{i} \}_{i=1}^m$ be the set of observed sequences. To infer $\mathcal{X^*}$ from noisy data, Maximum Likelihood (\textsf{ML}) estimator picks matrix $\hat{\mathcal{X}}$, which gives the highest probability $\mathrm{P}(\mathcal{M} \mid \hat{\mathcal{X}})$. We illustrate the space of $N \times \mathrm{L}$ binary matrices by $\{ 0 , 1 \} ^ {\mathrm{N} \mathrm{L}}$. The region of observations where \ML estimator makes the right decision can be represented by  
\begin{align}
\mathcal{A}_{m} = \{ \mathcal{M}  \ |  \mathrm{P}(\mathcal{M} \mid \mathcal{X}^*) > \mathrm{P}(\mathcal{M} \mid \mathcal{X}); \ \forall \mathcal{X} \neq \mathcal{X}^* \in \{ 0,1 \}^{\mathrm{N}\mathrm{L}} \}. \nonumber
\end{align}
Given matrices $ \mathcal{X}_{1} , \mathcal{X}_{2} $, We say $\mathcal{X}_{1}$ is equal to $\mathcal{X}_{2}$ and write $\mathcal{X}_{1} = \mathcal{X}_{2}$, if rows of $\mathcal{X}_{1}$ are a permutation of rows of $
\mathcal{X}_{2}$. 
Our analysis is independent of the order of rows, since re-permuting the rows in a matrix $\mathcal{X}$ does not change the distribution $P_{\X}$. Throughout the paper, we don't distinguish between matrices with same multiset of rows and different orders.

Note that $\Pr(A^c_{m})$ is \textsf{ML}'s probability of error. Our goal is to find the best exponent of error probability, defined as 
$$\mathcal{D}_{\mathcal{X}^*} = -\lim_{m \rightarrow \infty} \frac{1}{m}\log(\mathrm{P}(\mathcal{A}^c_{m})).$$  
We are interested in answering the following fundamental question: For a matrix $\tilde{\mathcal{X}} \neq \mathcal{X}^*$, what is the probability that $\tilde{\mathcal{X}}$ obtains a higher likelihood than $\mathcal{X}^*$?

In the hypothesis testing problem, we want to decide between two candidate distributions $P_{1}$ and $P_{2}$, based on observed sample vector $\{ x_{i} \}_{i=1}^m$. Let $P^{m}_{1} , P^{m}_{2}$ be the joint distributions of $m$ samples independently driven from $P_{1}$ and $P_{2}$ respectively. From Neyman-Pearson lemma \cite{thomas1991cover}, the optimal test has the rejection region $\mathrm{B}_{m} = \{ \frac{P^{(m)}_{2}(x)}{P^{(m)}_{1}(x)} \geq \mathrm{T}\}$, for any constant $\mathrm{T}$. Furthermore, for a fixed $\mathrm{T}$, we have 
\begin{align}
-\frac{1}{m}\lim_{m \rightarrow \infty} \log{P^{(m)}_{1}(\mathrm{B}_{m}}) & = -\frac{1}{m}\lim_{m \rightarrow \infty} \log{P^{(m)}_{2}(\mathrm{B}^c_{m}})   \nonumber \\ 
& = \mathcal{C}(P_{1},P_{2}), \nonumber
\end{align}
where $\mathcal{C}$ is the Chernoff information between $P_{1}$ and $P_{2}$, defined by 
\begin{align}
\mathcal{C}(P_{1} , P_{2}) = -\min_{0 \leq \lambda \leq 1} \log(\sum_{x} P_{1}^{\lambda}(x)P_{2}^{1-\lambda}(x) ). \nonumber
\end{align}

For desired matrix $\tilde{\mathcal{X}}$, let us define 
$$\mathcal{G}_{m}(\tilde{\mathcal{X}}) = \{ \mathcal{M} \ |  \ \Pr(\mathcal{M} | \tilde{\mathcal{X}}) \geq \Pr(\mathcal{M} | \mathcal{X}^*) \}.$$
Hence,
\begin{align}
& \max_{\tilde{\mathcal{X}} \neq \mathcal{X}^*}{\mathrm{P}(\mathcal{G}_{m}(\tilde{\mathcal{X}}))} \leq \mathrm{P}(A^c_{m}) = \mathrm{P}(\bigcup_{\tilde{\mathcal{X}} \neq \mathcal{X}^*} \mathcal{G}_{m}(\tilde{\mathcal{X}})) \nonumber \\ 
& \leq \sum_{\tilde{\mathcal{X}} \neq \mathcal{X}^*} \mathrm{P}(\mathcal{G}_{m}(\tilde{\mathcal{X}})) \leq 2^{\mathrm{N}\mathrm{L}} \max_{\tilde{\mathcal{X}} \neq \mathcal{X}^*}{\mathrm{P}(\mathcal{G}_{m}(\tilde{\mathcal{X}}))}, \nonumber
\end{align}
which yields
\begin{align}
\frac{1}{m} \log(\max_{\tilde{\mathcal{X}} \neq \mathcal{X}^*}{\mathrm{P}(\mathcal{G}_{m}(\tilde{\mathcal{X}}))}) \leq \frac{1}{m} \log( \mathrm{P}(\mathcal{A}^c_{m}))  & \leq \nonumber \\ 
\frac{\mathrm{N}\mathrm{L}}{m} + \frac{1}{m}\log(\max_{\tilde{\mathcal{X}} \neq \mathcal{X}^*}{\mathrm{P}(\mathcal{G}_{m}(\tilde{\mathcal{X}}))}). \nonumber
\end{align}
Thus,
\begin{align}
\mathcal{D}_{\mathcal{X}^*} & = -\lim_{m \rightarrow \infty}[\frac{1}{m} \log( \mathrm{P}(\mathcal{A}^c_{m}))]  \nonumber \\
& = -\lim_{m \rightarrow \infty}[ \frac{1}{m} \log( \max_{\tilde{\mathcal{X}} \neq \mathcal{X}^*}{\mathrm{P}(\mathcal{G}_{m}(\tilde{\mathcal{X}}))})] \nonumber \\
& =  \min_{\tilde{\mathcal{X}} \neq \mathcal{X}^*}  -\lim_{m \rightarrow \infty}[ \frac{1}{m} \log({\mathrm{P}(\mathcal{G}_{m}(\tilde{\X}))})] \nonumber \\
& =  \min_{\tilde{\X} \neq \X^*}  C(P_{\tilde{\X}} , P_{\X^*}). \nonumber
\end{align}

We are interested in finding the worst $\mathcal{X}^*$, where \ML  obtains its highest error asymptotically. Hence, if we define
\begin{align}
& \mathcal{D}_{\mathrm{worst}} = \min_{\mathcal{X}^*} \mathcal{D}_{\mathcal{X}^*} =  \min_{\mathcal{X}^* \in \{0 , 1\}^{\mathrm{N}\mathrm{L}}} \min_{ \tilde{\mathcal{X}} \neq \mathcal{X}^*} \mathcal{C}(P_{\tilde{\mathcal{X}}} , P_{\mathcal{X}^*}), 	\nonumber \\
& \mathcal{C}^*(\mathrm{N},\mathrm{L}) = \min_{\substack{\mathcal{X}_{1} , \mathcal{X}_{2} \in \{0 , 1 \}^{\mathrm{N}\mathrm{L}}, \\ \mathcal{X}_{1} \neq \mathcal{X}_{2}}}{\mathcal{C}(P_{\mathcal{X}_{1}} , P_{\mathcal{X}_{2}})}, \label{eq:minimization}
\end{align}
we have \
$\mathcal{D}_{\mathrm{worst}} = \mathcal{C}^*(\mathrm{N},\mathrm{L}). \nonumber$
This implies that in order to find the worst possible exponent of error with respect to \textsf{ML}, we need to find the closest pair of distributions in the set $\{ P_{\mathcal{X}} \ | \mathcal{X} \in \{ 0,1 \}^{\mathrm{N}\mathrm{L}} \}$ with regards to the measure of \CI. Hence, we aim to solve the minimization problem of \eqref{eq:minimization}.
\section{Main Results}\label{seq:main}
Our main result is stated in the followingTheorem. 
\begin{thm} \label{thm:main}
For $\mathcal{C}^*(\mathrm{N},\mathrm{L})$ defined in equation \eqref{eq:minimization}, we have
 \begin{enumerate}
\item [1)] For $\mathrm{f} \leq \frac{1}{4}$ and odd $\mathrm{N}$,
\begin{align}
\mathcal{C}^*(\mathrm{N},\mathrm{L}) = -\log(\sqrt{1 - \eta_{\mathrm{N}}^2}) \ , \ \eta_{\mathrm{N}} = \frac{1-2\mathrm{f}}{\mathrm{N}}. \nonumber
\end{align}
\item [2)] For $\mathrm{f} \leq \frac{1}{4}$ and even $\mathrm{N}$,
\begin{align}
 \!\!\!\!\!\!\!\!\!\!\! -\log(\sqrt{1 - \eta_{\mathrm{N}}^2}) \leq  \mathcal{C}^*(\mathrm{N},\mathrm{L}) \nonumber 
  \leq -\log(\frac{1}{\mathrm{N}} + \frac{\mathrm{N}-1}{\mathrm{N}}\sqrt{1-\eta_{\mathrm{N}-1}^2}).
\end{align}
\item [3)] Define $\mathcal{L} = \min(\mathrm{L} , \lfloor \log \mathrm{N} \rfloor + 1).$
Furthermore, define non-negative integers $\mathrm{k}$ and $\mathrm{R}$, where $\mathrm{k}=2n+1$, and
$\mathrm{N} = 2^{\mathcal{L}-1} \mathrm{k} + \mathrm{R}, \ \ \mathrm{R} < 2^{\mathcal{L}}$. Let  
\begin{align}
\ \ \ \ \epsilon_{\mathcal{L},\mathrm{N}} = \frac{[2(1-2\mathrm{f})]^{\mathcal{L}}}{2\mathrm{N}}. \nonumber
\end{align}
Then, for $\mathrm{f} > \frac{1}{4}$, we have
\begin{align}
&-\log(\sqrt{1 - \epsilon_{\mathcal{L},\mathrm{N}}^2}) \leq \mathcal{C}^*(\mathrm{N},\mathrm{L}) \nonumber \\    & \leq  -\log(\sqrt{(\frac{\mathrm{N}-\mathrm{R}}{\mathrm{N}})^2 - \epsilon_{\mathcal{L},\mathrm{N}}^2} + \frac{\mathrm{R}}{\mathrm{N}}).  \nonumber
\end{align}
\end{enumerate}
\end{thm}
The bounds represented in section 3 of Theorem \ref{thm:main} are tight in two regimes, which are summarized in Corollary \ref{cor:equalityconditions}.
\begin{cor} \label{cor:equalityconditions}
In the last section of Theorem \ref{thm:main}, equality holds ($R$ gets zero), iff one of the following conditions is satisfied.
\begin{enumerate}
\item [1)] $\mathrm{N}$ is a power of $2$, and $\mathrm{N} \leq 2^{\mathrm{L}-1}$.
\item [2)] $\mathrm{N} = 2^{\mathrm{L}-1}(2n + 1)$  for positive integer $n$. 
\end{enumerate} 
\end{cor}




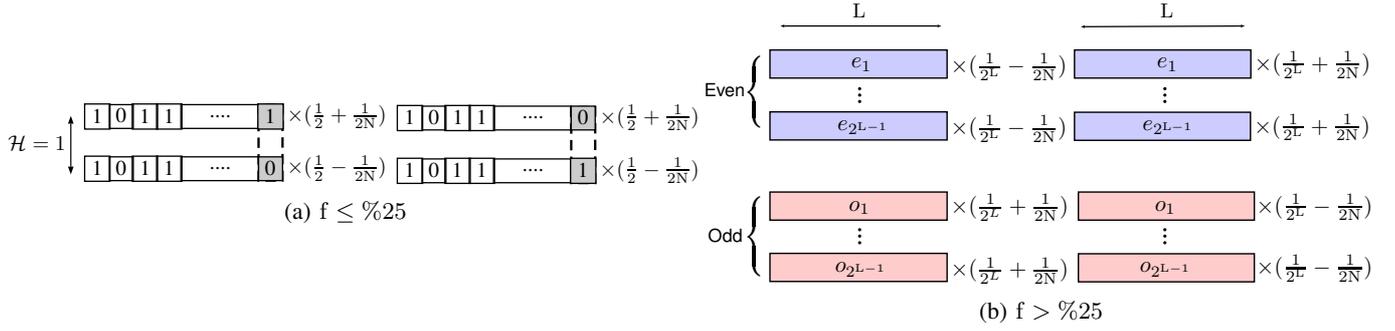
\begin{figure*}[ht]

\begin{subfigure}{0.5\textwidth}
%
%
\psscalebox{0.8 0.8} 
{
\begin{pspicture}(0,-0.67024124)(11.27479,0.67024124)
\definecolor{colour0}{rgb}{0.8,0.8,0.8}
\rput[bl](4.5847898,-0.61124134){$ \times (\frac{1}{2} - \frac{1}{2\mathrm{N}}) $}
\psframe[linecolor=black, linewidth=0.038, linestyle=dashed, dash=0.17638889cm 0.10583334cm, dimen=middle](4.509169,0.6482411)(4.114927,-0.61124134)
\psframe[linecolor=black, linewidth=0.022, dimen=inner](1.6609871,0.6482411)(1.2667447,0.25399867)
\psframe[linecolor=black, linewidth=0.022, dimen=middle](2.0552297,0.6482411)(1.6609871,0.25399867)
\psframe[linecolor=black, linewidth=0.022, dimen=inner](2.8437145,0.6482411)(2.449472,0.25399867)
\rput[bl](1.3747672,0.34743413){1}
\rput[bl](1.7690096,0.34743413){0}
\rput[bl](2.5574944,0.34743413){1}
\psframe[linecolor=black, linewidth=0.022, dimen=middle](4.114927,0.6482411)(2.837957,0.25399867)
\rput[bl](3.3096085,0.4379128){....}
\psframe[linecolor=black, linewidth=0.022, fillstyle=solid,fillcolor=colour0, dimen=inner](4.509169,0.6482411)(4.114927,0.25399867)
\psframe[linecolor=black, linewidth=0.022, dimen=inner](1.6609871,-0.21699892)(1.2667447,-0.61124134)
\psframe[linecolor=black, linewidth=0.022, dimen=middle](2.0552297,-0.21699892)(1.6609871,-0.61124134)
\psframe[linecolor=black, linewidth=0.022, dimen=inner](2.8437145,-0.21699892)(2.449472,-0.61124134)
\rput[bl](1.3747672,-0.5178059){1}
\rput[bl](1.7690096,-0.5178059){0}
\rput[bl](2.163252,-0.5178059){1}
\rput[bl](2.5574944,-0.5178059){1}
\psframe[linecolor=black, linewidth=0.022, dimen=middle](4.114927,-0.21699892)(2.837957,-0.61124134)
\rput[bl](3.3096085,-0.42732728){....}
\psframe[linecolor=black, linewidth=0.022, fillstyle=solid,fillcolor=colour0, dimen=middle](4.509169,-0.21699892)(4.114927,-0.61124134)
\rput[bl](4.2223735,0.34743413){1}
\rput[bl](4.2223735,-0.5178059){0}
\rput[bl](4.5847898,0.25399867){$ \times (\frac{1}{2} + \frac{1}{2\mathrm{N}}) $}
\psframe[linecolor=black, linewidth=0.022, dimen=inner](2.4437144,0.6482411)(2.049472,0.25399867)
\rput[bl](2.1574945,0.34743413){1}
\psframe[linecolor=black, linewidth=0.022, dimen=inner](2.4437144,-0.21699892)(2.049472,-0.61124134)
\rput[bl](9.72479,0.22875863){$ \times (\frac{1}{2} + \frac{1}{2\mathrm{N}}) $}
\psframe[linecolor=black, linewidth=0.038, linestyle=dashed, dash=0.17638889cm 0.10583334cm, dimen=middle](9.649169,0.60824114)(9.254927,-0.65124136)
\psframe[linecolor=black, linewidth=0.022, dimen=inner](6.8009872,-0.25175887)(6.4067445,-0.64600134)
\psframe[linecolor=black, linewidth=0.022, dimen=middle](7.1952295,-0.25175887)(6.8009872,-0.64600134)
\psframe[linecolor=black, linewidth=0.022, dimen=inner](7.9837146,-0.25175887)(7.589472,-0.64600134)
\rput[bl](6.514767,-0.5525659){1}
\rput[bl](6.9090095,-0.5525659){0}
\rput[bl](7.6974945,-0.5525659){1}
\psframe[linecolor=black, linewidth=0.022, dimen=middle](9.254927,-0.25175887)(7.977957,-0.64600134)
\rput[bl](8.449609,-0.4620872){....}
\psframe[linecolor=black, linewidth=0.022, fillstyle=solid,fillcolor=colour0, dimen=inner](9.649169,-0.25175887)(9.254927,-0.64600134)
\psframe[linecolor=black, linewidth=0.022, dimen=inner](6.8009872,0.6230011)(6.4067445,0.22875863)
\psframe[linecolor=black, linewidth=0.022, dimen=middle](7.1952295,0.6230011)(6.8009872,0.22875863)
\psframe[linecolor=black, linewidth=0.022, dimen=inner](7.9837146,0.6230011)(7.589472,0.22875863)
\rput[bl](6.514767,0.3221941){1}
\rput[bl](6.9090095,0.3221941){0}
\rput[bl](7.303252,0.3221941){1}
\rput[bl](7.6974945,0.3221941){1}
\psframe[linecolor=black, linewidth=0.022, dimen=middle](9.254927,0.6230011)(7.977957,0.22875863)
\rput[bl](8.449609,0.41267273){....}
\psframe[linecolor=black, linewidth=0.022, fillstyle=solid,fillcolor=colour0, dimen=middle](9.649169,0.6230011)(9.254927,0.22875863)
\rput[bl](9.362373,-0.5525659){1}
\rput[bl](9.362373,0.3221941){0}
\rput[bl](9.72479,-0.64600134){$ \times (\frac{1}{2} - \frac{1}{2\mathrm{N}}) $}
\psframe[linecolor=black, linewidth=0.022, dimen=inner](7.5837145,-0.25175887)(7.189472,-0.64600134)
\rput[bl](7.2974944,-0.5525659){1}
\psframe[linecolor=black, linewidth=0.022, dimen=inner](7.5837145,0.6230011)(7.189472,0.22875863)
\psline[linecolor=black, linewidth=0.02, arrowsize=0.05291667cm 2.0,arrowlength=1.4,arrowinset=0.0]{<->}(1.0285715,0.46252683)(1.0285715,-0.5089017)
\rput[bl](0.0,-0.10890174){$\mathcal{H} = 1$}
\end{pspicture}
}

\caption{$\mathrm{f} \leq \% 25$}
\label{fig:subim1}
\end{subfigure}
\hfill
\begin{subfigure}{0.5\textwidth}
\psscalebox{0.85 0.85} 
{
\begin{pspicture}(0,-2.1879456)(10.256145,2.1879456)
\definecolor{colour1}{rgb}{0.8,0.8,1.0}
\definecolor{colour2}{rgb}{1.0,0.8,0.8}
\psframe[linecolor=black, linewidth=0.022, fillstyle=solid,fillcolor=colour1, dimen=outer](3.7675939,1.4769223)(1.0045596,1.016673)
\psframe[linecolor=black, linewidth=0.022, fillstyle=solid,fillcolor=colour1, dimen=outer](3.7675939,0.5112368)(1.0045596,0.050987557)
\psframe[linecolor=black, linewidth=0.022, fillstyle=solid,fillcolor=colour2, dimen=outer](3.7675939,-0.7393863)(1.0045596,-1.1996355)
\psframe[linecolor=black, linewidth=0.022, fillstyle=solid,fillcolor=colour2, dimen=outer](3.7675939,-1.6963176)(1.0045596,-2.1565669)
\rput{90.0}(1.1555473,-0.6393478){\rput[bl](0.8974475,0.2580997){\small $\overbrace{ \ \ \ \ \ \ \ \  \ \ }$ \par}}
\rput[bl](2.2714198,1.1411108){$e_{1}$}
\rput[bl](2.0617507,0.14985591){$e_{2^{\mathrm{L}-1}}$}
\rput[bl](2.237327,-1.0990626){$o_{1}$}
\rput[bl](2.022265,-2.064748){$o_{2^{\mathrm{L}-1}}$}
\rput[bl](3.8124874,0.9911036){$\times (\frac{1}{2^\mathrm{L}} - \frac{1}{2\mathrm{N}})$}
\rput[bl](3.8124874,0.025418155){$\times (\frac{1}{2^\mathrm{L}} - \frac{1}{2\mathrm{N}})$}
\rput[bl](2.2956922,1.9721363){\sffamily \small{$\mathrm{L}$}}
\rput[bl](7.056327,1.9779456){\small \sffamily $\mathrm{L}$ \par}
\psframe[linecolor=black, linewidth=0.022, fillstyle=solid,fillcolor=colour1, dimen=outer](8.46706,1.4769223)(5.7166424,1.016673)
\psframe[linecolor=black, linewidth=0.022, fillstyle=solid,fillcolor=colour1, dimen=outer](8.461251,0.50542754)(5.7166424,0.045178294)
\rput[bl](7.0047846,1.1411108){$e_{1}$}
\rput[bl](6.780053,0.14985591){$e_{2^{\mathrm{L}-1}}$}
\rput[bl](8.511953,1.0252943){$\times (\frac{1}{2^\mathrm{L}} + \frac{1}{2\mathrm{N}})$}
\rput[bl](8.506145,0.05379963){$\times (\frac{1}{2^\mathrm{L}} + \frac{1}{2\mathrm{N}})$}
\rput{90.0}(-1.1607614,-2.9556563){\rput[bl](0.8974475,-2.0582087){$\overbrace{ \ \ \ \ \ \ \ \  \ \ }$}}
\psdots[linecolor=black, dotsize=0.05](7.1248274,0.86036533)
\psdots[linecolor=black, dotsize=0.05](7.1248274,0.7603653)
\psdots[linecolor=black, dotsize=0.05](7.1248274,0.6603653)
\rput[bl](3.8124874,-1.205205){$\times (\frac{1}{2^L} + \frac{1}{2\mathrm{N}})$}
\rput[bl](3.8124874,-2.1879456){$\times (\frac{1}{2^L} + \frac{1}{2\mathrm{N}})$}
\psframe[linecolor=black, linewidth=0.022, fillstyle=solid,fillcolor=colour2, dimen=outer](8.521251,-0.7393863)(5.7766423,-1.1996355)
\psframe[linecolor=black, linewidth=0.022, fillstyle=solid,fillcolor=colour2, dimen=outer](8.521251,-1.7021269)(5.7766423,-2.1623762)
\rput[bl](6.970692,-1.0990626){$o_{1}$}
\rput[bl](6.740567,-2.064748){$o_{2^{\mathrm{L}-1}}$}
\rput[bl](8.5461445,-1.1910142){$\times (\frac{1}{2^\mathrm{L}} - \frac{1}{2\mathrm{N}})$}
\psline[linecolor=black, linewidth=0.02, arrowsize=0.013cm 1.67,arrowlength=1.34,arrowinset=0.0]{<->}(1.1767576,1.8399218)(3.6683562,1.8399218)
\psline[linecolor=black, linewidth=0.02, arrowsize=0.013cm 1.67,arrowlength=1.34,arrowinset=0.0]{<->}(5.8248277,1.828658)(8.363257,1.8399218)
\psdots[linecolor=black, dotsize=0.05](7.1248274,-1.3396347)
\psdots[linecolor=black, dotsize=0.05](7.1248274,-1.4396347)
\psdots[linecolor=black, dotsize=0.05](7.1248274,-1.5396347)
\psdots[linecolor=black, dotsize=0.05](2.4048276,0.86036533)
\psdots[linecolor=black, dotsize=0.05](2.4048276,0.7603653)
\psdots[linecolor=black, dotsize=0.05](2.4048276,0.6603653)
\rput[bl](0.06,-1.5320317){\sffamily \footnotesize{Odd}}
\psdots[linecolor=black, dotsize=0.05](2.4048276,-1.3396347)
\psdots[linecolor=black, dotsize=0.05](2.4048276,-1.4396347)
\psdots[linecolor=black, dotsize=0.05](2.4048276,-1.5396347)
\rput[bl](8.5461445,-2.1510143){$\times (\frac{1}{2^\mathrm{L}} - \frac{1}{2\mathrm{N}})$}
\rput[bl](0.0,0.7479683){\sffamily \footnotesize{Even}}
\end{pspicture}
}
\caption{$\mathrm{f} > \% 25$}
\label{fig:subim2}
\end{subfigure}
 
\caption{Phase transition in the closest pair of source sequences with respect to their frequencies. (a) The noise level is below \% 25. For odd $N$, the closest pairs consist of two sequences with Hamming distance 1, and frequencies $\frac{1}{2} - \frac{1}{2N}$ and $\frac{1}{2} + \frac{1}{2N}$. (b) The  noise level exceeds \% 25, for $N = 2^L (2n+1)$. The closest pairs changes to the case where sources have all sequences with length $L$, having a $\frac{1}{N}$ difference between frequencies of the sequences in $U_{odd}$ in return to $U_{even}$.}
\label{fig:image2}
\end{figure*}


In Fig. \ref{fig:image2},  the closest pair of sources is illustrated with respect to their frequencies in two tight cases regarding $N = 2n+1$ for $\mathrm{f} \leq \frac{1}{4}$, and $\mathrm{N} = 2^\mathrm{L}(2n+1)$ for $\mathrm{f} > \frac{1}{4}$. Surprisingly, there exists a phase transition in the source structure when the noise level exceeds \%25. 

Let $\mathcal{U}_{\mathrm{even}}$ and $\mathcal{U}_{\mathrm{odd}}$ be the sets consisting of the sequences with length $\mathrm{L}$, which have even and odd number of ones, respectively. For noise probability less than \%25, the closest pair is expressed by sources which have two types of sequences with Hamming distance one, and frequencies $\frac{1}{2} - \frac{1}{2\mathrm{N}}$ and $\frac{1}{2} + \frac{1}{2\mathrm{N}}$. However, when the noise probability exceeds \%25, a deformation happens in the space of distributions corresponding to the sources, such that the closest pair incredibly alters to a totally different case; The two closest sources have all sequences of length $\mathrm{L}$, with $\frac{1}{\mathrm{N}}$ discrepancy between frequencies of sequences in $\mathcal{U}_{\mathrm{even}}$ versus $\mathcal{U}_{\mathrm{odd}}$.  It's worth to mention the astounding phase transition in reduction speed of \CI regarding the closest pair, when noise probability exceeds the threshold. The order of reduction changes from linear decrease, to polynomial decrease with degree $\mathrm{L}$.

\section{Proof Ideas}\label{seq:proof}
In this paper, we sketch the proof of Theorem \ref{thm:main}, by presenting the main ideas and results. We only consider the cases where our bounds are tight. In particular, we only focus on the first part of Theorem \ref{thm:main}, and the second part of Corollary \ref{cor:equalityconditions}. The reader can find the complete proofs of the expressed lemmas and theorems among with the proof of other cases of Theorem \ref{thm:main} in the full version of the paper \cite{mypaper}. 

The main idea behind the proof is to provide a method of lower bounding on the \CI, such that \textsf{CI} between every unequal pair of matrices $\mathcal{X}_{1} , \mathcal{X}_{2}$ can be lower bounded by a simple value $\tau$ as $$\tau \leq \mathcal{C}(P_{\mathcal{X}_{1}} , P_{\mathcal{X}_{2}}).$$ 

Our lower bounding technique arrives at $\tau$ by performing $\mathrm{L}$ iterations of column reduction on $\mathcal{X}_{1}$ and $\mathcal{X}_{2}$. At the first step, we transform $\mathcal{X}_{1} , \mathcal{X}_{2}$ to $\tilde{\mathcal{X}_{1}} , \tilde{\mathcal{X}_{2}}$ with $\mathrm{L}-1$ columns, such that 
\begin{align}
\mathcal{C}(P_{\tilde{\mathcal{X}_{1}}} , P_{\tilde{\mathcal{X}_{2}}}) \leq \mathcal{C}(P_{\mathcal{X}_{1}} , P_{\mathcal{X}_{2}}). \nonumber
\end{align}
Continuing iteratively, we reach to one-dimensional \textsf{BMMs}, where we can lower bound the \CI quite easily. Surprisingly, we can find specific pairs of matrices $(\mathcal{X}^*_{1} , \mathcal{X}^*_{2})$, such that column reductions do not incur any loss in terms of \CI. This means that $\mathcal{C}(\mathcal{X}^*_{1} , \mathcal{X}^*_{2})$ is the lower bound on any $\mathcal{C}(\mathcal{X}_{1},\mathcal{X}_{2})$. It is worth mentioning that $\mathcal{X}^*_{1}$ and $\mathcal{X}^*_{2}$ are functions of $\mathrm{f}$ which we elaborate on next.\\

\subsection{Definitions}
The idea behind the definition of column reduction is based on the concept of critical columns. Given matrices $\mathcal{X}_{1} \neq \mathcal{X}_{2}$ with $\mathrm{L}^{\prime}$ columns, the pair of $\ell$th columns in $\mathcal{X}_{1}$ and $\mathcal{X}_{2}$ is \textit{critical}, if by eliminating them, matrices become equal. 
Moreover, we call $(\mathcal{X}_{1} , \mathcal{X}_{2})$ a \textit{critical} pair, if for each $1 \leq \ell \leq \mathrm{L^{\prime}}$, the $\ell$th pair of columns in $\mathcal{X}_{1}$ and $\mathcal{X}_{2}$ is critical. 

For a desired $\mathrm{N} \times \mathrm{L}^{\prime}$ matrix $\X$, let us consider $\mathrm{f}_{\ell}$ as a specific flip rate corresponding to the $\ell$th column. Hence, the distribution $P_{\X}$ has the flip probability $\mathrm{f}_{\ell}$ with respect to the $\ell$th entry of each row in $\X$. By assumption, in the beginning of the reduction process, all columns have flip probability $\mathrm{f}$. However, as we will see, the reductions can change the flip probabilities. 
\begin{dif}
 For each $\mathrm{L^{\prime}} \leq \mathrm{L}$ and $1 \leq i < j \leq \mathrm{L^{\prime}}$, define the map $\phi^{\mathrm{L^{\prime}}}_{i,j} : \{ 0,1 \}^{\mathrm{NL^{\prime}}} \rightarrow \{ 0,1 \}^{\mathrm{N(L^{\prime}-1)}}$, such that for an $\mathrm{N} \times \mathrm{L^{\prime}}$ matrix $\mathcal{X}$, $\phi^{\mathrm{L^{\prime}}}_{i,j}(\mathcal{X})$ is obtained by removing the $i$th and $j$th columns $\mathcal{X}^{(i)} , \mathcal{X}^{(j)}$ and replacing $\mathcal{X}^{(i)} \oplus \mathcal{X}^{(j)}$ as a new column, with a flip probability defined as 
\begin{align}
\mathrm{f}_{\mathrm{new}} = \mathrm{f}_{i}(1 - \mathrm{f}_{j}) + (1 - \mathrm{f}_{i})\mathrm{f}_{j}. \label{eq:fnew}
\end{align} 
Furthermore, define $\mathbb{g}(\mathrm{f}) = 1-2\mathrm{f}$. Equation \eqref{eq:fnew} implies 
\begin{align}
\mathbb{g}(\mathrm{f}_{\mathrm{new}}) = \mathbb{g}(\mathrm{f}_{i})\mathbb{g}(\mathrm{f}_{j}). \label{eq:g}
\end{align}

For any column $\ell$ with flip rate $\mathrm{f}_{\ell}$, $\mathbb{g}(\mathrm{f}_{\ell}) \in [0,1]$ is a measure of the $\ell$th column's informativeness. From equation \eqref{eq:g}, we obtain $\mathbb{g}(\mathrm{f}_{\mathrm{new}}) \leq \mathbb{g}(\mathrm{f}_{i})$  and $\mathbb{g}(\mathrm{f}_{\mathrm{new}}) \leq \mathbb{g}(\mathrm{f}_{j})$, which means that merging two columns by $\phi$ 
decreases their informativeness.
\end{dif}
\begin{dif} \label{def:delta}
Given matrices $\mathcal{X}_{1} , \mathcal{X}_{2}$, define $\delta({\mathcal{X}_{1}}) , \delta({\mathcal{X}_{2}})$ to be the matrices obtained by iteratively removing two equal rows from both $\mathcal{X}_{1}$ and $\mathcal{X}_{2}$. Two rows are equal if their corresponding entries are equal. Hence, $\delta(\mathcal{X}_{1})$ and $\delta(\mathcal{X}_{2})$ don't share any equal rows. Define $\mathcal{S}_{1} , \mathcal{S}_{2}$ as the set of indices corresponding to the removed rows from $\X_{1} , \X_{2}$ respectively. For a non-negative integer $\mathrm{t}$, the pair $( \mathcal{X}_{1} , \mathcal{X}_{2} )$ has ${t}$ degrees of regularity, if rows in each of $\delta(\mathcal{X}_{1})$ and $\delta(\mathcal{X}_{2})$ can be partitioned into a set of clusters, where each cluster has exactly $2^\mathrm{t}$ elements, and all of the rows in each cluster are equal to one another. Since rows in any matrix can be partitioned into clusters of size $2^0=1$., the degree of regularity for any pair of matrices is at least zero.\end{dif}
\subsection{Column Reductions}

\begin{lem} [Column Elimination]\label{lem:chernoffloss}
Omitting a non-critical pair of columns from $\mathcal{X}_{1}$ and $\mathcal{X}_{2}$ results in an unequal pair $( \mathcal{X}^{'}_{1} , \mathcal{X}^{'}_{2} )$, where 
\begin{align}
\mathcal{C}(P_{\mathcal{X}^{'}_{1}} , P_{\mathcal{X}^{'}_{2}}) \leq \mathcal{C}(P_{\mathcal{X}_{1}} , P_{\mathcal{X}_{2}}). \label{ineq:first}
\end{align}
The equality holds if the eliminated columns are identical, having either all zero or all one entries. Note that by the definition, every pair of non-critical matrices $\mathcal{X}_{1},\mathcal{X}_{2}$ has at least one non-critical pair of columns.
\end{lem}

\begin{lem}[Column Merging] \label{lem:chernoffloss2}
Given a critical pair of $\mathrm{N} \times \mathrm{L}^{\prime}$ matrices $( \mathcal{X}_{1} , \mathcal{X}_{2} )$, for every $1 \leq i < j \leq \mathrm{L^{\prime}} \leq \mathrm{L}$, $( \phi^{\mathrm{L^{\prime}}}_{i,j}(\mathcal{X}_{1}) , \phi^{\mathrm{L^{\prime}}}_{i,j}(\mathcal{X}_{2}) )$ is also a critical pair. Furthermore, if $( \mathcal{X}_{1},\mathcal{X}_{2} )$ has at least $t$ degrees of regularity, then $( \phi^{\mathrm{L^{\prime}}}_{i,j}(\mathcal{X}_{1}) ,\phi^{\mathrm{L^{\prime}}}_{i,j}(\mathcal{X}_{2}) )$ has at least $t+1$ degrees of regularity, and we have
\begin{align}
\mathcal{C}(P_{\phi^{\mathrm{L^{\prime}}}_{i,j}(\mathcal{X}_{1})} , P_{\phi^{\mathrm{L^{\prime}}}_{i,j}(\mathcal{X}_{2})}) \leq \mathcal{C}(P_{\mathcal{X}_{1}} , P_{\mathcal{X}_{2}}). \label{ineq:second}
\end{align}
Merging reduction does not change sets $\mathcal{S}_{1}$ and $\mathcal{S}_{2}$ designated in Definition \ref{def:delta}. Moreover, a sufficient condition for equality to hold is that there exists a permutation $\pi$ on rows of $\mathcal{X}_{2}$ and a partitioning $\Re$ of $\{1, \ldots , \mathrm{N} \}$ into pairs, such that:
\begin{enumerate}
\item [1)] For each pair $( s , r ) \in \Re$ and every $k \neq i , j$
\begin{equation*}
\!\!\!\!\!\!\! \mathcal{X}_{1}(s,k) =  \mathcal{X}_{1}(r,k) 
= \mathcal{X}_{2}(\pi(s),k) = \mathcal{X}_{2}(\pi(r),k). \nonumber
\end{equation*} 
\item[2)]  For each pair $( s , r ) \in \Re$
\begin{gather}
\!\!\!\!\!\!\! \mathcal{X}_{1}(s,i) \neq \mathcal{X}_{1}(r,i), \mathcal{X}_{2}(\pi(s) , i) \neq \mathcal{X}_{2}(\pi(r) , i)  \nonumber, \\
\!\!\!\!\!\!\! \mathcal{X}_{1}(s,j) \neq \mathcal{X}_{1}(r,j),
\mathcal{X}_{2}(\pi(s) , j) \neq \mathcal{X}_{2}(\pi(r) , j). \nonumber
\end{gather}
\end{enumerate}
\end{lem}

Note that both of our reductions preserve the inequality assumption on the matrices.  

 \subsection{Proof Sketch}
For an unequal pair $\{ \mathcal{X}_{1} , \mathcal{X}_{2}\}$, we apply reduction by eliminating non-critical pairs of columns, until we arrive at a critical pair. Suppose $\alpha$ columns are eliminated in this phase, and the remaining columns are $\{ \ell_{1} , ... , \ell_{\mathrm{L}-\alpha}\}$. Note that $\alpha(\mathcal{X}_{1},\mathcal{X}_{2})$ is a function of $\mathcal{X}_{1}$ and $\mathcal{X}_{2}$. Now we apply reduction by merging two columns in each step. At the end, $\mathrm{L} = 1$, and we have two unequal one-dimensional \textsf{BMM}s, where clusters are represented by our one-dimensional matrices $\mathcal{X}^{\mathrm{br}}_{1}$ and $\mathcal{X}^{\mathrm{br}}_{2}$, with flip rate $\mathrm{f}_{\mathrm{br}}$. Hence, for $\mathrm{b} \in \{ 0 , 1\}$ we have
\begin{align}
\!\!\!\!\! \mathrm{P}(\mathrm{b} \mid \mathcal{X}^{\mathrm{br}}_{u}) = \frac{\sum_{i=1}^{\mathrm{N}} \mathrm{f}_{\mathrm{br}}^{\mathcal{X}^{\mathrm{br}}_{u,i} \oplus b} (1-\mathrm{f}_{\mathrm{br}})^{\overline{\mathcal{X}^{\mathrm{br}}_{u,i} \oplus \mathrm{b}}}}{\mathrm{N}}, \ \ u \in \{ 1,2 \} , \label{eq:bernoulli}
\end{align}
where $\mathcal{X}^{\mathrm{br}}_{u,i}$ is the $i$th element of $\mathcal{X}^{\mathrm{br}}_{u}$.  Thus, $\mathrm{P}(\mathrm{b} \mid \mathcal{X}^{\mathrm{br}}_{1})$ and $\mathrm{P}(\mathrm{b} \mid \mathcal{X}^{\mathrm{br}}_{2})$ are Bernoulli distributions, with parameters denoted by $p_{\mathrm{br},1}$ and $p_{\mathrm{br},2}$. Furthermore, equation \eqref{eq:g} reveals
\begin{align}
\mathbb{g}(\mathrm{f}_{\mathrm{br}}) = \prod_{i = 1}^{\mathrm{L}-\alpha} \mathbb{g}(\mathrm{f}_{\ell_{i}})  = (1 - 2\mathrm{f})^{\mathrm{L} - \alpha}. \label{eq:productg}
\end{align}
 Note that $\mathrm{f}_{\mathrm{br}} = \frac{1 - \mathbb{g}(\mathrm{f}_{\mathrm{br}})}{2}$. Therefore, according to the definition of $\phi$, we conclude that the resulted matrices and flip rates are independent of the order of merging.
Note that we have merged $\mathrm{L} - \alpha$ columns into one column. Hence, by Lemma \ref{lem:chernoffloss2}, the pair $\{ \mathcal{X}^{\mathrm{br}}_{1} , \mathcal{X}^{\mathrm{br}}_{2} \}$ has $\mathrm{L} - \alpha - 1$ degrees of regularity. According to equation (\ref{eq:bernoulli}), this implies that there exist non-negative integers $a_{1} , a_{2} , c_{1} , c_{2}$, with $a_{1} + c_{1} = a_{2} + c_{2}$, such that for $u \in \{ 1,2 \}$,
\begin{align*}
 & p_{\mathrm{br},u} = \mathrm{P}(\mathrm{b} \mid \mathcal{X}^{\mathrm{br}}_{u})  = \\
 &\frac{\sum_{i \notin \mathcal{S}_{u}}^{\mathrm{N}} \mathrm{f}_{\mathrm{br}}^{\mathcal{X}^{\mathrm{br}}_{u,i} \oplus b} (1-\mathrm{f}_{\mathrm{br}})^{\overline{\mathcal{X}^{\mathrm{br}}_{u,i} \oplus \mathrm{b}}}}{\mathrm{N}}
  +  \frac{\sum_{i \in \mathcal{S}_{u}}^{\mathrm{N}} \mathrm{f}_{\mathrm{br}}^{\mathcal{X}^{\mathrm{br}}_{u,i} \oplus b} (1-\mathrm{f}_{\mathrm{br}})^{\overline{\mathcal{X}^{\mathrm{br}}_{u,i} \oplus \mathrm{b}}}}{\mathrm{N}} \\
& = a_{u}\frac{2^{\mathrm{L} - \alpha}\mathrm{f}_{\mathrm{br}}}{\mathrm{N}} + c_{u}\frac{2^{\mathrm{L}-\alpha}(1-\mathrm{f}_{\mathrm{br}})}{\mathrm{N}} + \mathbb{C}. \nonumber
\end{align*}
 where $\mathbb{C}$ is a constant, independent of $u$. Since $\mathcal{X}^{\mathrm{br}}_{1} \neq \mathcal{X}^{\mathrm{br}}_{2}$ we have $p_{\mathrm{br},1} \neq p_{\mathrm{br},2}$, which yeilds
\begin{align}
| p_{\mathrm{br},1} - p_{\mathrm{br},2} | \geq \frac{2^{\mathrm{L}-\alpha}(1-2\mathrm{f}_{\mathrm{br}})}{\mathrm{N}} =  \frac{2^{\mathrm{L}-\alpha}\mathbb{g}(\mathrm{f}_{br})}{\mathrm{N}}. \nonumber
\end{align}
Regarding inequality \eqref{eq:productg} we obtain
\begin{align}
| p_{\mathrm{br},1} - p_{\mathrm{br},2} | \geq \frac{[2(1-2\mathrm{f})]^{\mathrm{L}-\alpha}}{2\mathrm{N}} \label{eq:l1-distance}.
\end{align}

 The above inequality reveals a lower bound on $\mathbb{L}_1$ distance of Bernoulli's $\mathrm{P}(b \mid \mathcal{X}^{\mathrm{br}}_{1})$ and $\mathrm{P}(b \mid \mathcal{X}^{\mathrm{br}}_{2})$, which is a function of the number of elimination and merging reductions applied to $\{ \mathcal{X}_{1} , \mathcal{X}_{2} \}$. The following lemma finds the minimum \CI between two Bernoulli random variables, given that their $\mathbb{L}_{1}$ distance is lower bounded by $2\epsilon$.
 \begin{lem} \label{lem:bernoulli}
Given a pair of Bernoulli distributions with probabilities $p$ and $q$, define $\mathcal{C}_{\mathrm{br}}(p,q)$ to be the \CI between them. Suppose we have $|p - q| \geq \epsilon$. Then,
\begin{align}
\mathcal{C}_{\mathrm{br}}(p,q) \geq \mathcal{C}_{\mathrm{br}}(\frac{1-\epsilon}{2} , \frac{1 + \epsilon}{2}) = -\log(\sqrt{1 - \epsilon^2}) \label{eq:bernoulli2}.
\end{align}
\end{lem}
Combining equation \eqref{eq:l1-distance} with Lemma.\ref{lem:bernoulli}, and regarding the fact that reductions do not increase \CI, we obtain a lower bound on \CI as 
\begin{align}
\mathcal{LB}(\mathcal{X}_{1} , \mathcal{X}_{2}) = \mathcal{C}_{br}(\frac{1-\wp}{2} , \frac{1+\wp}{2} ) = -\log(\sqrt{1 - \wp^2}), \nonumber
\end{align} 
where
\begin{align}
\wp =  \frac{[2(1-2\mathrm{f})]^{L-\alpha}}{2\mathrm{N}}. \label{eq:wp}
\end{align}
Now, instead of minimizing $\mathcal{C}(P_{\mathcal{X}_{1}} , P_{\mathcal{X}_{2}})$, we seek to minimize the lower bound $\mathcal{LB}(\mathcal{X}_{1} , \mathcal{X}_{2})$. To this end, we have to minimize $\wp$ with respect to $\mathcal{X}_{1}$ and $\mathcal{X}_{2}$, for which based on the value of $\mathrm{f}$, there are two cases.
\subsubsection{Case 1: $\mathrm{f} \leq \frac{1}{4}$}
In this regime we have $2(1-2\mathrm{f}) \geq 1$. Thus, according to equation \eqref{eq:wp}, we have to minimize $\mathrm{L} - \alpha$, which leads to $\alpha = \mathrm{L}-1$. \label{eq:alpha} Therefore, 
\begin{align}
\tau_{1} = \min_{\mathcal{X}_{1} \neq \mathcal{X}_{2}} \mathcal{LB}(\mathcal{X}_{1} , \mathcal{X}_{2}) =  -\log(\sqrt{1 - \eta_{\mathrm{N}}^2}) \ , \ \eta_{\mathrm{N}} = \frac{1-2\mathrm{f}}{\mathrm{N}}. \nonumber
\end{align}
For every $\mathcal{X}_{1} \neq \mathcal{X}_{2}$, $\mathcal{C}(P_{\mathcal{X}_{1}} , P_{\mathcal{X}_{2}})$ is lower bounded by $\tau_{1}$.

Equation $\alpha = \mathrm{L}-1$ points out that the pair  $\mathcal{X}^*_{1} , \mathcal{X}^*_{2}$ which minimizes $\mathcal{LB}$ should have $\mathrm{L}-1$ non-critical columns, eliminated one by one iteratively. On the other hand, to illustrate the tightness of $\tau_{1}$, we should define $\mathcal{X}^*_{1} , \mathcal{X}^*_{2}$ in such a way that $\mathcal{C}(P_{\mathcal{X}^*_{1}} , P_{\mathcal{X}^*_{2}}) = \mathcal{LB}(\mathcal{X}^*_{1} , \mathcal{X}^*_{2})$. This implies that all the inequalities \eqref{ineq:first},\eqref{eq:l1-distance},\eqref{eq:bernoulli2} should turn into equalities.
\begin{lem} \label{lem:example1}
Suppose $\mathrm{N} = 2n + 1$. Consider two sequences $\upsilon_{1} , \upsilon_{2}$ with length $\mathrm{L}$, and Hamming distance one. Define matrices $\mathcal{X}^*_{1}$ and $\mathcal{X}^*_{2}$, such that $\mathcal{X}^*_{1}$ has $n$ replicas of $\upsilon_{1}$ and $n+1$ replicas of $\upsilon_{2}$ as its rows, while $\mathcal{X}^*_{2}$ has $n$ replicas of $\upsilon_{2}$ and $n+1$ replicas of $\upsilon_{1}$. Then, for defined $\mathcal{X}^*_{1}$ and $\mathcal{X}^*_{2}$, $\mathcal{C}(P_{\mathcal{X}^*_{1}} , P_{\mathcal{X}^*_{2}})$ meets the lower bound $\tau_{1}$.
\end{lem}
\subsubsection{Case 2: $\mathrm{f} > \frac{1}{4}$}
We attain $2(1-2\mathrm{f}) < 1$. Hence, we have to maximize $\mathrm{L} - \alpha$ which results in $\alpha = 0$. Therefore,
\begin{align}
 \tau_{2} = \min_{\mathcal{X}_{1} \neq \mathcal{X}_{2}}\mathcal{LB}(\mathcal{X}_{1} , \mathcal{X}_{2}) = -\log(\sqrt{1 - \epsilon_{\mathrm{L},\mathrm{N}}^2}), \nonumber
\end{align}
where
$\epsilon_{\mathrm{L},\mathrm{N}} = \frac{[2(1-2\mathrm{f})]^{\mathrm{L}}}{2\mathrm{N}}. \nonumber$
For every $\mathcal{X}_{1} \neq \mathcal{X}_{2}$, $\mathcal{C}(P_{\mathcal{X}_{1}} , P_{\mathcal{X}_{2}})$ is lower bounded by $\tau_{2}$.

Equation $\alpha = 0$ states that the pair  $\mathcal{X}_{1} , \mathcal{X}_{2}$  which minimizes $\mathcal{LB}$ should be critical. Moreover, in order to show the tightness of $\tau_{2}$, we need to find a critical pair $\mathcal{X}^*_{1} , \mathcal{X}^*_{2}$, such that the inequalities \eqref{ineq:second},\eqref{eq:l1-distance},\eqref{eq:bernoulli2} turn into equalities. 
\begin{lem} \label{lem:sufficientcondition}
$( \mathcal{X}_{1} , \mathcal{X}_{2} )$ is a critical pair if and only if there exists a number $n^*$, such that rows of one of $\delta(\mathcal{X}_{1})$ or $\delta(\mathcal{X}_{2})$ consist of $n^*$ replicas of each sequence in $\mathcal{U}_{\mathrm{even}}$, while rows of the other one consist of $n^*$ replicas of each sequence in $\mathcal{U}_{\mathrm{odd}}$. Furthermore, if we apply merging reductions on $\{ \mathcal{X}_{1} , \mathcal{X}_{2} \}$, a sufficient condition on $\mathcal{X}_{1}, \mathcal{X}_{2}$  to incur  no information loss in all of the reduction steps is that there exist integers $n_{1} $ and $n_{2}$, such that rows of $\mathcal{X}_{1}$ consist of $n_{1}$ replicas of  each sequence in $\mathcal{U}_{\mathrm{even}}$ and $n_{2}$ replicas of each sequence in $\mathcal{U}_{\mathrm{odd}}$, while rows of $\mathcal{X}_{2}$ consist of $n_{2}$ replicas of each sequence in $\mathcal{U}_{\mathrm{even}}$, and $n_{1}$ replicas of each sequence in $\mathcal{U}_{\mathrm{odd}}$. 
\end{lem}
\begin{lem}  \label{lem:example2}
Suppose N = $2^{L-1}(2n+1)$, for a non-negative integer $n$. Define matrix $\X^*_{1}$ to have $n$ replicas of each sequence in $U_{even}$ and $n+1$ replicas of each sequence in $U_{odd}$ as its rows. Similarly, define $\X^*_{2}$ to have $n+1$ replicas of each sequence in $U_{even}$ and $n$ replicas of each sequence in $\mathcal{U}_{\mathrm{odd}}$. Then, $\mathcal{C}(P_{\mathcal{X}^*_{1}} , P_{\mathcal{X}^*_{2}})$ meets the lower bound $\tau_{2}$.
\end{lem} 
\subsection{Generalized Theorem} \label{thm:generalized}
We generalize our result to the case 
where we have a vector of parameters $\mathrm{F} = \{\mathrm{f}_{\ell}\}_{\ell=1}^\mathrm{L}$, such that the $\ell$th entry of each source symbol is flipped with probability $\mathrm{f}_{\ell}$. Hence, the $\ell$th column of $\mathcal{X^*}$ has flip probability $\mathrm{f}_{\ell}$. Similarly, define $\mathcal{C}^*(\mathrm{N},\mathrm{L},\mathrm{F}) = \min_{\mathcal{X}_{1} \neq \mathcal{X}_{2}} \mathcal{C}(P_{\mathcal{X}_{1}} , P_{\mathcal{X}_{2}})$.
\begin{thm} \label{thm:generalized}
Regarding above notations, let $\gamma = \left\vert \{  \mathrm{f}_{i} \mid \mathrm{f}_{i} > \frac{1}{4}\} \right\vert $.   Define $\mathcal{K} = \min(\gamma , \lfloor \log \mathrm{N} \rfloor + 1)$. In addition, define non-negative integers $\mathrm{k},\mathrm{R}$, where $\mathrm{k}=2n+1$, and 
$$\mathrm{N} = 2^{\mathcal{K}-1}\mathrm{k} + \mathrm{R} , \ \mathrm{R} < 2^{\mathcal{K}}.$$ 
Furthermore, let $\{  \mathrm{f}_{\lambda_{i}} \}_{i = 1}^{\mathcal{K}}$
 be the $\mathcal{K}$ largest flip rates. Let,
\begin{align}
\epsilon_{\mathcal{K}} = \frac{2^{\mathcal{K} - 1} \prod_{i=1}^{\mathcal{K}} (1-2\mathrm{f}_{\lambda_{i}})}{\mathrm{N}}. \nonumber
\end{align}
Then, we have
\begin{align}
  -\log(\sqrt{1 - \epsilon_{\mathcal{K}}^2}) \leq \mathcal{C}^*(\mathrm{N},\mathrm{L},\mathrm{F}) \nonumber \leq -\log(\sqrt{(\frac{\mathrm{N}-\mathrm{R}}{\mathrm{N}})^2 - \epsilon_{\mathcal{K}}^2} + \frac{\mathrm{R}}{\mathrm{N}}).  
\end{align}
\end{thm}

\begin{cor} \label{cor:equalityconditions2}
In the previous theorem, equality occurs iff one of the conditions bellow takes place.
\begin{enumerate}
\item [1)] $N$ is a power of $2$ and $N \leq 2^{\gamma-1}$.
\item [2)] $N = 2^{\gamma-1}(2n + 1)$  for positive integer $n$. 
\end{enumerate}
\end{cor}


\section{Conclusion}\label{seq:con}
We have obtained an asymptotically tight upper bound for the  \textsf{ML} estimator in Binary Mixture Identification.  Our findings shows an amazing phase transition in the discrete space of distributions. When the noise level exceeds 0.25, a severe reduction in the minimum \CI distance is observed. We proposed a systematic procedure to tightly bound the \CI, which might be useful for bounding \CI in other desired spaces and probably extendable to continues spaces. Addressing the worst-case scenario, it would be of great interest to attain bounds for any pair of sources based on our methodologies.

\appendix
\begin{lem} \label{lem:basic}
For positive real numbers $a , b , c , d$ and for any $0 \leq \lambda \leq 1$, we have:
$$a^{\lambda}b^{1-\lambda} + c^{\lambda}d^{1-{\lambda}} \leq (a + c)^{\lambda}(b + d)^{1 - \lambda}$$ 
\end{lem}

\begin{proof}
From the convexity of log function,  for any $0 \leq \lambda \leq 1$, we have 
$$\lambda \log x + (1-\lambda) \log y \leq \log ( \lambda x + (1- \lambda)y).$$
This implies 
$$x^{\lambda}y^{1-\lambda} \leq \lambda x + (1- \lambda)y.$$

Setting $x =\frac{a}{a+c}$ and $y=\frac{b}{b+d}$ yields
$$(\frac{a}{a+c})^{\lambda} (\frac{b}{b + d})^{1-\lambda}  \leq \lambda(\frac{a}{a+c}) + (1-\lambda)(\frac{b}{b+d}). $$
Similarly, setting $x =\frac{a}{a+c}$ and $y=\frac{b}{b+d}$ yields
$$ (\frac{c}{a+c})^{\lambda}(\frac{d}{b+d})^{1-\lambda} \leq \lambda(\frac{c}{a+c}) + (1-\lambda)(\frac{d}{b+d}) .$$
Adding the two inequalities gives the desired inequality. 
\end{proof}

\begin{dif}[Match quadruple] \label{def:quadruple}
Given matrices $\X_{1}$ and $\X_{2}$, rows $s_{1},r_{1}$ from $\delta(\X_{1})$ and $s_{2},r_{2}$ from $\delta(\X_{2})$ are called an (i,j)-match quadruple, if
\begin{gather}
\!\!\!\!\!\!\!\!\! \delta(\mathcal{X}_{1})(s_{1},i) = \delta(\mathcal{X}_{2})(s_{2},i), \delta(\mathcal{X}_{1})(s_{1}, j) \neq \delta(\mathcal{X}_{2})(s_{2} , j)  \nonumber, \\
\!\!\!\!\!\!\!\!\! \delta(\mathcal{X}_{1})(r_{1},i) = \delta(\mathcal{X}_{2})(r_{2},i), \delta(\mathcal{X}_{1})(r_{1}, j) \neq \delta(\mathcal{X}_{2})(r_{2} , j)  \nonumber, \\
\!\!\!\!\!\!\!\!\! \delta(\mathcal{X}_{1})(s_{1},j) = \delta(\mathcal{X}_{2})(r_{2},j), \delta(\mathcal{X}_{1})(s_{1}, i) \neq \delta(\mathcal{X}_{2})(r_{2} , i)  \nonumber, \\
\!\!\!\!\!\!\!\!\! \delta(\mathcal{X}_{1})(s_{2},j) = \delta(\mathcal{X}_{2})(r_{1},j), \delta(\mathcal{X}_{1})(s_{2}, i) \neq \delta(\mathcal{X}_{2})(r_{1} , i)  \nonumber.
\end{gather}
For each column $k \neq i , j$:
\begin{equation*}
\!\!\!\!\!\!\! \delta(\mathcal{X}_{1})(s_{1},k) =  \delta(\mathcal{X}_{1})(s_{2},k) 
= \delta(\mathcal{X}_{2})(r_{1},k) = \delta(\mathcal{X}_{2})(r_{2},k). \nonumber
\end{equation*} 
\end{dif}

\begin{dif}
Let $\alpha=(\alpha_{1} , \cdots , \alpha_{L})$. For matrices $X_{1}$ and $X_{2}$ define 
\begin{equation*} 
f_{\lambda} (\X_{1} , \X_{2}) = \sum_{\alpha \in   \{ 0 , 1 \} ^ {\Ll}}   {P_{\X_{1}}(\alpha)}^{\lambda}{P_{\X_{2}}(\alpha)}^{1-\lambda},
\end{equation*}
for which we have
\begin{equation}
\mathcal{C}(P_{\X_{1}} , P_{\X_{2}}) = -\min_{\lambda \in [0,1]} \log(f_{\lambda} (\X_{1} , \X_{2})). \label{eq:chernoffdefinition}
\end{equation}
\end{dif}

\begin{proof}[\textbf{Proof of Lemma \ref{lem:chernoffloss}}]
Suppose that we arrive at matrices $\X^{\prime}_{1} , \X^{\prime}_{2}$ by removing the $i$th columns from $\X_{1} , \X_{2}$. Let $\bar \alpha = (\alpha_{1} , ... , \alpha^c_{i} , ... , \alpha_{\N})$ be the $n-1$ dimensional vector obtained by removing $\alpha_{i}$ from vector $\alpha$. For each $1 \leq \lambda \leq 1$, 
\begin{align}
& f_{\lambda}(\X_{1} , \X_{2})  = 
\sum_{\bar \alpha \in \{ 0 , 1 \}^{\Ll-1}} \sum_{\alpha_{\L} \in \{ 0 , 1 \}} {P_{\X_{1}}(\alpha)}^{\lambda}{P_{\X_{2}}(\alpha)}^{1-\lambda}  \stackrel{ \text{(a)}}{\leq} \nonumber \\
& \sum_{ \bar \alpha  \in \{ 0 , 1 \}^{\Ll-1}} \!\!\!\!\!\!\!\! \left[ P_{\X_{1}}(\overline{\alpha},0) + P_{\X_{1}}(\overline{\alpha},1)  \right]^{\lambda} \left[  P_{\X_{2}}(\overline{\alpha},0) + P_{\X_{2}}(\overline{\alpha},1)  \right]^{1 - \lambda} \nonumber \\
& = \sum_{\bar{\alpha}  \in \{ 0 , 1 \}^{\Ll-1}} \left[ P_{\X_{1}^{\prime}}(\bar{\alpha}) \right]^{\lambda} \left[  P_{ \X_{2}^{\prime}}(\bar{\alpha})  \right]^{1 - \lambda} \nonumber \\
& =  f_{\lambda} ( \X^{\prime}_1 , \X^{\prime}_2). \label{eq:colomitt} 
\end{align}
where inequality $(a)$ follows from Lemma $\ref{lem:basic}$. Combining the above inequality with equation \eqref{eq:chernoffdefinition}, inequality \eqref{ineq:first} is proved.

For the equality to happen, we need all the  inequalities in $(a)$ to become equality. This can be achieved if 
\begin{align}
\frac{P_{\X_{1}}(\overline{\alpha},0)}{P_{\X_{1}}(\overline{\alpha},1)} = \frac{P_{\X_{2}}(\overline{\alpha},0)}{P_{\X_{2}}(\overline{\alpha},1)}, \nonumber
\end{align}
for all $\bar{\alpha} \in \{ 0 , 1 \}^{\Ll-1}$. It can be readily shown that the condition will be met if the removed columns are identical and  have either all zero or all one entries. By this, the proof is accomplished.
\end{proof}

\begin{lem} \label{lem:partitioning}
For a critical pair $(\X_{1} , \X_{2})$, for given $1 \leq i < j \leq \Ll$, there exist a disjoint partition $\mathcal{T}$ of rows in $\delta({\X_{1}})$ and $\delta({\X_{2}})$ into (i,j)-match quadruples. 

\begin{proof}
According to the definition of critical pairs, for each $1 \leq i \leq \Ll$, removing the $i$th columns from both of $\delta(\X_{1}) , \delta(\X_{2})$ gives matrices $\delta^{\prime}(\X_{1}) = \delta^{\prime}(\X_{2})$. 
We refer to the permutation which maps the rows in $\delta^{\prime}(\X_{1})$ to rows in $\delta^{\prime}(\X_{2})$ as $\Pi^{(i)}$. 

For fixed $1 \leq i \leq j \leq \Ll$, given a desired row $s_{1}$ in $\delta(\X_{1})$, define 
$$s_{2} = \Pi^{(i)}(s_{1}) , r_{2} = \Pi^{(j)}(s_{1}) , r_{1} = (\Pi^{(i)})^{-1}(r_{2}).$$
By definition, $\delta(\X_{1}) , \delta(\X_{2})$ does not have any common row. Thus,
\begin{align}
& \delta(\X_{2})(s_{2} , i) = \overline{\delta(\X_{1})(s_{1} , i)} = \overline{\delta(\X_{2})(r_{2} , i)} = \delta(\X_{1})(r_{1} , i),  \label{eq:pi1}  \\ 
& \delta(\X_{2})(s_{2} , j) = \delta(\X_{1})(s_{1} , j) = \overline{\delta(\X_{2})(r_{2} , j)} = \overline{\delta(\X_{1})(r_{1} , j)}, \label{eq:pi2} \\
& \forall k \neq i , j: \nonumber \\
& \delta(\X_{2})(s_{2} , k) = \delta(\X_{1})(s_{1} , k) = \delta(\X_{2})(r_{2} , k) = \delta(\X_{1})(r_{1} , k) \label{eq:pi3} ,
\end{align}
 From equation \eqref{eq:pi3}, we obtain $\Pi^{(j)}(s_{2}) = r_{1}$. Hence, we can divide rows in $\delta(\X_{1}) , \delta(\X_{2})$ into quadruples, such that for each quadruple consisting of rows $s_{1} , r_{1}$ in $\delta(\X_{1})$, and $s_{2} , r_{2}$ in $\delta(\X_{2})$, we have
\begin{gather}
\Pi^{(i)}(s_{1}) = r_{1} , \Pi^{(j)}(s_{1}) =  r_{2} , \Pi^{(i)}(s_{2}) = r_{2} , \Pi^{(j)}(s_{2}) = r_{1}. \nonumber
\end{gather} 
This property, combined with equations \eqref{eq:pi1} and \eqref{eq:pi2}, directly yield to equations in Definition \ref{def:quadruple}, regarding (i,j)-match quadruples, which completes the proof of Lemma \ref{lem:partitioning}.
\end{proof}
\end{lem}

\begin{proof}[\textbf{Proof of Lemma \ref{lem:chernoffloss2}}]
Let $\bar \alpha = (\alpha_{1} , ... , \alpha^c_{i} , ... ,  \alpha^c_{j} , ... , \alpha_{\N})$ be the $n-2$ dimensional vector obtained by removing $\alpha_{i} , \alpha_{j}$ from vector $\alpha$. We have
\begin{align*}
& f_{\lambda}(\X_{1} , \X_{2})  = 
\sum_{\bar \alpha \in \{ 0 , 1 \}^{\Ll-2}} \sum_{\substack{\mathrm{b} \in \{ 0 , 1 \} \\ \alpha_{i} \oplus \alpha_{j} =  \mathrm{b} }} {P_{\X_{1}}(\alpha)}^{\lambda}{P_{\X_{2}}(\alpha)}^{1-\lambda} \nonumber \\
& \stackrel{ \text{(b)}}{\leq} \sum_{\bar \alpha} \sum_{\mathrm{b} \in \{ 0 , 1 \}} [{ \sum_{\substack{\alpha_{i} \oplus \alpha_{j} =  \mathrm{b} }} P_{\X_{1}}(\alpha)}]^{\lambda} [{ \sum_{\substack{\alpha_{i} \oplus \alpha_{j} =  \mathrm{b} }} P_{\X_{2}}(\alpha)}]^{1-\lambda} \\
& = \sum_{\bar \alpha} \sum_{\mathrm{b} \in \{ 0 , 1 \}} [P_{\phi^\Ll_{i,j}(\X_{1})}(\bar \alpha , \mathrm{b})]^{\lambda} [P_{\phi^\Ll_{i,j}(\X_{2})}(\bar \alpha , \mathrm{b})]^{1-\lambda} \\
& = \sum_{\tilde{\alpha} \in \{0,1 \}^{\Ll - 1}} [P_{\phi^\Ll_{i,j}(\X_{1})}(\tilde \alpha)]^{\lambda} [P_{\phi^\Ll_{i,j}(\X_{2})}(\tilde \alpha)]^{1-\lambda} \\
& = f_{\lambda}(\phi^\Ll_{i,j}(\X_{1}) , \phi^\Ll_{i,j}(\X_{2})),
\end{align*}
where $(b)$ follows from Lemma \ref{lem:basic}. Combining the above inequality with equation \eqref{eq:chernoffdefinition} yields to inequality \eqref{ineq:second}. 

It can be readily checked that if the equality assumption of Lemma \ref{lem:chernoffloss2} is satisfied, then for $\Ll$ dimensional vectors $\alpha$ and $\alpha^{\prime}$, which have equal entries in each index $k \neq i , j$, and also $\alpha_{i} \oplus \alpha_{j} = \alpha^{\prime}_{i} \oplus \alpha^{\prime}_{j}$, we have
$$P_{\X_{u}}(\alpha) = P_{\X_{u}}(\alpha^{\prime}), \ u \in \{ 1,2\}.$$
This satisfies the equality condition of all the inequalities applied in $(b)$, and leads to an equality case of Lemma \ref{lem:chernoffloss2} as desired.

We index the new column obtained from merging columns $i,j$ by $\ell_{new}$. According to Lemma \ref{lem:partitioning}, we can partition rows of $\delta(\X_{1})$ and $\delta(\X_{2})$ into (i,j)-match quadruples. Consider a quadruple consisting of rows $s_{1} , r_{1}$ from $\delta(\X_{1})$ and $s_{2} , r_{2}$ from $\delta(\X_{2})$. If we name $\phi^\Ll_{i,j}(\delta(\X_{1}))$ as $\mathcal{D}_{1}$ and $\phi^\Ll_{i,j}(\delta(\X_{2}))$ as $\mathcal{D}_{2}$, by the properties of match quadruples, for each column $k \neq \ell_{new}$: 
\begin{align}
\mathcal{D}_{1}(s_{1} , k) = \mathcal{D}_{1}(r_{1} , k) = 
 \mathcal{D}_{2}(s_{2} , k) = \mathcal{D}_{2}(r_{2} , k). \label{eq:clusters0}
\end{align}
In addition,
\begin{gather}
\Big ( \D_{1}(s_{1} , \ell_{new}) = \D_{1}(r_{1} , \ell_{new}) \Big) \nonumber \\ 
\neq \Big ( \D_{2}(s_{2} , \ell_{new}) = \D_{2}(r_{2} , \ell_{new}) \Big ). \label{eq:clusters}
\end{gather}
We call such a quadruple of rows $(s_{1} , r_{1} , s_{2} , r_{2})$ regarding $(\D_{1} , \D_{2} )$, which is obtained from an (i,j)-match quadruple of $(\delta(\X_{1}) , \delta(\X_{2}))$, a \textit{new-(i,j)-match quadruple}. We claim that $\D_{1} , \D_{2}$ does not share any common row. 

Suppose that rows $s_{1}$ from $\D_{1}$, and $s^{\prime}_{2}$ from $\D_{2}$ are equal. From the proof of Lemma \ref{lem:partitioning}, rows in $\delta(\X_{1}) , \delta(\X_{2})$ can be partitioned into (i,j)-match quadruples. Hence, we deduce that rows in $\D_{1} , \D_{2}$ can be partitioned into new-(i,j)-match quadruples. Let $\mathcal{Q}_{1} = (s_{1} , s_{2} , r_{1} , r_{2})$ and $\mathcal{Q}_{2} = (s^{\prime}_1 , r^{\prime}_{1} , s^{\prime}_{2} , r^{\prime}_{2})$ be two new-(i,j)-match quadruples in $\D_{1} , \D_{2}$, which contain rows $s_{1}$ and $s^{\prime}_{2}$ respectively. Due to equation \eqref{eq:clusters}, $\mathcal{Q}_{1}$ and $\mathcal{Q}_{2}$ must be distinct. On the other hand, by equations \eqref{eq:clusters0} and \eqref{eq:clusters}, we perceive 
\begin{gather*}
\D_{1}(s_{1}) = \D_{1}(r_{1}) = \D_{2}(s^{\prime}_{2}) = \D_{2}(r^{\prime}_{2}), \\
\D_{1}(s_{2}) = \D_{1}(r_{2}) = \D_{2}(s^{\prime}_{1}) = \D_{2}(r^{\prime}_{1}),
\end{gather*}
where $\D_{u}(n) \ (u \in \{ 1,2 \})$ is the $n$th row of $\D_{u}$. It can be readily checked that for the corresponding (i,j)-match quadruples $(s_{1} , s_{2} , r_{1} , r_{2})$ and $(s^{\prime}_1 , r^{\prime}_{1} , s^{\prime}_{2} , r^{\prime}_{2})$ in $(\delta(\X_{1}) , \delta(\X_{2}))$,
\begin{gather*}
\delta(\X_{1})(s_{1}) = \delta(\X_{2})(s^{\prime}_{2}) \ , \  \delta(\X_{1})(r_{1}) = \delta(\X_{2})(r^{\prime}_{2}), \\
\delta(\X_{1})(s^{\prime}_{1}) = \delta(\X_{2})(s_{2}) \ , \ \delta(\X_{1})(r^{\prime}_{1}) = \delta(\X_{2})(r_{2}).
\end{gather*}
This means that $\delta(\X_{1})$ and $\delta(\X_{2})$ must share common rows as well, which is a contradiction. Hence, matrices $\D_{1} , \D_{2}$ don't have any common rows. Therefore, we conclude for the pair of matrices $(\phi^\Ll_{i,j}(\X_{1}) , \phi^\Ll_{i,j}(\X_{2}))$, 
$$\delta(\phi^\Ll_{i,j}(\X_{1})) = \D_{1} \ , \ \delta(\phi^\Ll_{i,j}(\X_{2})) = \D_{2}.$$ 
This yields $\phi^\Ll_{i,j}(\X_{1}) \neq \phi^\Ll_{i,j}(\X_{2})$. On the other hand, according to the definition of $\phi$, we have that for every $1 \leq \ell \leq \Ll-1$, removing $\ell$th columns from $\phi^\Ll_{i,j}(\X_{1})$ and $\phi^\Ll_{i,j}(\X_{2})$ results in equal matrices. Thus, we conclude that $(\phi^\Ll_{i,j}(\X_{1}) , \phi^\Ll_{i,j}(\X_{2}))$ is a critical pair as well. Furthermore, the set of common indices $\mathcal{S}_{1}$ and $\mathcal{S}_{2}$, designated in Definition \ref{def:delta}, does not change by applying merging reductions.
  
By definition of regularity, we know that rows in each of $\delta(\X_{1})$ and $\delta(\X_{2})$ can be partitioned into clusters of size $2^t$, with rows in each cluster are completely equal to each other. Hence, the existence of a partitioning of rows into (i,j)-match quadruples extends to a partitioning of clusters into match quadruples with the given properties. Equation \eqref{eq:clusters} reveals that by applying $\phi$, a given match quadruple consisting of clusters $C_{1} , C_{2}$ in $\delta(\X_{1})$ and $C_{3} , C_{4}$ in $\delta(\X_{2})$, turns into clusters $C^{\prime}_{12} , C^{\prime}_{34}$ in $\delta(\phi^\Ll_{i,j}(\X_{1}))$ and $\delta(\phi^\Ll_{i,j}(\X_{2}))$ respectively with size $2^{t+1}$, such that all of the rows in each of $C^{\prime}_{12}$ and $C^{\prime}_{34}$ are equal to one another. Hence, we conclude that $(\phi^\Ll_{i,j}(\X_{1}) , \phi^\Ll_{i,j}(\X_{2}))$ has at least $t+1$ degrees of regularity. 
\end{proof}

\begin{lem} \label{lem:symmetry}
For given distributions $P_{1} , P_{2}$ on binary sequences with length $L$, suppose that there exist a partitioning $\mathcal{V}$ of all sequences with length L into pairs, such that for each pair $( s_{1} , s_{2} ) \in \mathcal{V}$,
\begin{gather}
P_{1}(s_{1}) = P_{2}(s_{2}), \label{first} \\ 
P_{1}(s_{2}) = P_{2}(s_{1}). \label{second}
\end{gather}
Then
$$\mathcal{C}(P_{1} , P_{2}) = -\log \Big ( \sum_{\ell \in \{ 0,1 \}^\Ll} \sqrt{P_{1}(\ell)P_{2}(\ell)} \Big).$$
\end{lem} 
\begin{proof}
For a desired pair $(s_{1} , s_{2})$, we prove
\begin{align}
P_{1}(s_{1})^{\lambda}P_{2}(s_{1})^{1-\lambda} & + P_{1}(s_{2})^{\lambda}P_{2}(s_{2})^{1-\lambda} \nonumber \\
& \geq \sqrt{P_{1}(s_{1})P_{2}(s_{1})} + \sqrt{P_{1}(s_{2})P_{2}(s_{2})}. \label{ineq:atomi}
\end{align}
To this end, we obtain the derivative's root of LHS in \eqref{ineq:atomi}, with respect to $\lambda$.
\begin{align}
\log(\frac{P_{1}(s_{1})}{P_{2}(s_{1})}) & P_{1}(s_{1})^{\lambda^*}P_{2}(s_{1})^{1-\lambda^*}  + \nonumber \\ 
&\log(\frac{P_{1}(s_{2})}{P_{2}(s_{2})}) P_{1}(s_{2})^{\lambda^*}P_{2}(s_{2})^{1-\lambda^*} = 0. \label{derivative}
\end{align}
Equation \eqref{derivative} combined with \eqref{first} , \eqref{second} reveals
$\lambda^* = \frac{1}{2}$, which expresses inequality \eqref{ineq:atomi} as desired.

Summing inequality \eqref{ineq:atomi} for all pairs in $\mathcal{V}$, and taking minimum with respect to $\lambda$, we arrive at
\begin{align}
\min_{0 \leq \lambda \leq 1} \Big( \sum_{(s_{1} , s_{2}) \in \mathcal{V}} P_{1}(s_{1})^{\lambda}P_{2}(s_{1})^{1-\lambda} + P_{1}(s_{2})^{\lambda}P_{2}(s_{2})^{1-\lambda} \Big) \nonumber \\
\geq \sum_{(s_{1} , s_{2}) \in \mathcal{V}} \sqrt{P_{1}(s_{1})P_{2}(s_{1})} + \sqrt{P_{1}(s_{2})P_{2}(s_{2})}.\nonumber
\end{align}
Note that in the above inequality, the minimum of LHS is taken over all values of $0 \leq \lambda \leq 1$, which includes $\lambda = \frac{1}{2}$. Thus, we also have LHS $ \leq $ RHS, which yields LHS$=$ RHS. Hence, according to \eqref{eq:chernoffdefinition}, proof of Lemma \ref{lem:symmetry} is complete.
\end{proof}

\begin{lem} \label{lem:bernoullikomaki}
For a given positive real $\epsilon$, consider the family $F_{\epsilon}$ consisting of all pairs of Bernoulli distributions with parameters $p$ and $q$, such that $q-p = \epsilon$. Then
$$\min_{((p,1-p),(q,1-q)) \in F_{\epsilon}} \mathcal{C}_{br}(p,q) = \mathcal{C}_{br}(\frac{1-\epsilon}{2},\frac{1+\epsilon}{2}).$$
\end{lem}

\begin{proof}
Define 
$$f_{\lambda}(p,q) = p^{\lambda}q^{\lambda} + (1-p)^{1 - \lambda}(1-q)^{1 - \lambda}.$$
Without loss of generality, we assume $p \leq \frac{1}{2}$. Given $q - p = \epsilon$, we should have either $p$ or $1-q$ less that $\frac{1-\epsilon}{2}$. if the second inequality was the case, then we define the new variables $\tilde{p} = 1 -q$ , $\tilde{q} = 1 - p$. this way, we have $\tilde{p} \leq \frac{1-\epsilon}{2}$. Hence without loss of generality, We can also assume $0 \leq p \leq \frac{1 - \epsilon}{2}$. We calculate the condition under which the derivative of $f_{\lambda}(p,q)$ with respect to $\lambda$ is positive:
\begin{align}
& \frac{\text{d}(p^{\lambda}(p+\epsilon)^{1-\lambda} + (1-p)^{\lambda}(1-p-\epsilon)^{1-\lambda})}{\text{d} p} \geq 0 .  \nonumber
\end{align}
This is equevalent to
\begin{align}
\log(\frac{p}{1-p})(1-\lambda) + & \log(\frac{p+\epsilon}{1-(p+\epsilon)})\lambda \nonumber \\
& \leq \log(\frac{(1-\lambda)p + \lambda (p + \epsilon)}{1 - ((1-\lambda)p + \lambda (p + \epsilon) )}). \label{eq:jensen}
\end{align}

But note that the function $g(x) = \log(\frac{x}{1-x})$ is concave for $x \leq \frac{1}{2}$, and convex for $x \geq \frac{1}{2}$, because
\begin{align}
\frac{d^2 g}{d x^2} = \frac{-(1-2x)}{x^2(1-x)^2} = \begin{cases}< 0 &  x < \frac{1}{2}\\ \geq 0 & x \geq \frac{1}{2} \end{cases}. \nonumber
\end{align}
Hence if $p + \epsilon \leq \frac{1}{2}$, by Jensen's inequality, we obtain \eqref{eq:jensen} for every $1 \leq \lambda \leq 1$. Hence, for $0 \leq p_{1} \leq p_{2} \leq \frac{1}{2} - \epsilon$,
\begin{equation}
-\min_{0 \leq \lambda \leq 1} \log(f_{\lambda}(p_{1},q))  \leq  -\min_{0 \leq \lambda \leq 1} \log(f_{\lambda}(p_{2},q)). \label{eq:monotone1}
\end{equation}
In the following, we are going to prove that inequality \eqref{eq:jensen} also holds for $\frac{1}{2} - \epsilon \leq p \leq \frac{1-\epsilon}{2}$, when $\lambda \leq \frac{1}{2}$. In Figure.\ref{fig:jensenpic}, we can see the curve with respect to $g(x) = \log(\frac{x}{1-x})$. Let $\ell(x,y)$ be the line passing through desired points $x$ and $y$. The red line corresponds to $\ell((p , g(p)) , (p+\epsilon , g(p+\epsilon))$. Note that the inequality $(1-\lambda)p + \lambda (p + \epsilon) \leq \frac{1}{2}$ is true for $\lambda \leq \frac{1}{2}$. Thus, in order to prove inequality \eqref{eq:jensen} for $0 \leq \lambda \leq \frac{1}{2}$, it is sufficient to illustrate that $\ell((p , g(p)) , (p+\epsilon , g(p+\epsilon))$ is under the graph of $g$, in the interval $[p , \frac{1}{2}]$. 

For $0 \leq x \leq 1$, we have 
\begin{align}
\frac{d(g)}{dx} = \frac{1}{x(1-x)} \geq 0. \label{eq:gfunction}
\end{align}
Also, note that $p \leq \frac{1 - \epsilon}{2}$ yields $p + \epsilon \leq 1 - p$. This inequality, among with Inequality \eqref{eq:gfunction}, reveals that the slope of $\ell((p , g(p)) , (p+\epsilon , g(p+\epsilon))$ is not less than the slope of $\ell((p , g(p)) , (1-p , g(1-p))$, which can be seen as the blue line in Figure.\ref{fig:jensenpic}. Therefore, it is sufficient to prove that $\ell((p , g(p)) , (1-p , g(1-p))$ is under the graph of $g$. But note that $\ell((p , g(p)) , (1-p , g(1-p))$ passes through the point $(\frac{1}{2} , 0)$. Thus, Concavity of $g$ in $[0 , \frac{1}{2}]$ implies that $\ell((p , g(p)) , (1-p , g(1-p))$ is under the graph of $g$ in $[p,\frac{1}{2}]$. Hence, for $\frac{1 - \epsilon}{2} \leq p_{1} \leq p_{2} \leq \frac{1}{2} - \epsilon$, we get that
\begin{align}
-\min_{0 \leq \lambda \leq \frac{1}{2}} \log(f_{\lambda}(p_{1},q))  \leq  -\min_{0 \leq \lambda \leq \frac{1}{2}} \log(f_{\lambda}(p_{2},q)). \label{eq:monotone2}
\end{align}
Next we prove that for $\frac{1}{2} - \epsilon \leq p \leq \frac{1 - \epsilon}{2}$, $f_{\lambda}(p,q)$ takes its minimum with respect to $\lambda \in [0,1]$, in a point $\lambda^{*} \leq \frac{1}{2}$. To this end, we calculate the derivative's root of $f_{\lambda}(p,q)$, with respect to $\lambda$.

\begin{align}
& \frac{\text{d}(p^{\lambda}(p+\epsilon)^{1-\lambda} + (1-p)^{\lambda}(1-p-\epsilon)^{1-\lambda})}{\text{d}\lambda}(\lambda^*) = 0 \nonumber
\end{align}
This is equivalent to
\begin{align}
\log(\frac{p}{p+\epsilon}) & p^{\lambda^{*}}(p+\epsilon)^{1-\lambda^{*}}  = \nonumber \\
& \log(\frac{1-(p+\epsilon)}{1-p}) (1-p)^{\lambda^{*}}(1-(p+\epsilon))^{1-\lambda^{*}}. \label{eq:jensen1}
\end{align}

 \begin{figure}[t]
   \centering
   \includegraphics[width=0.52\textwidth]{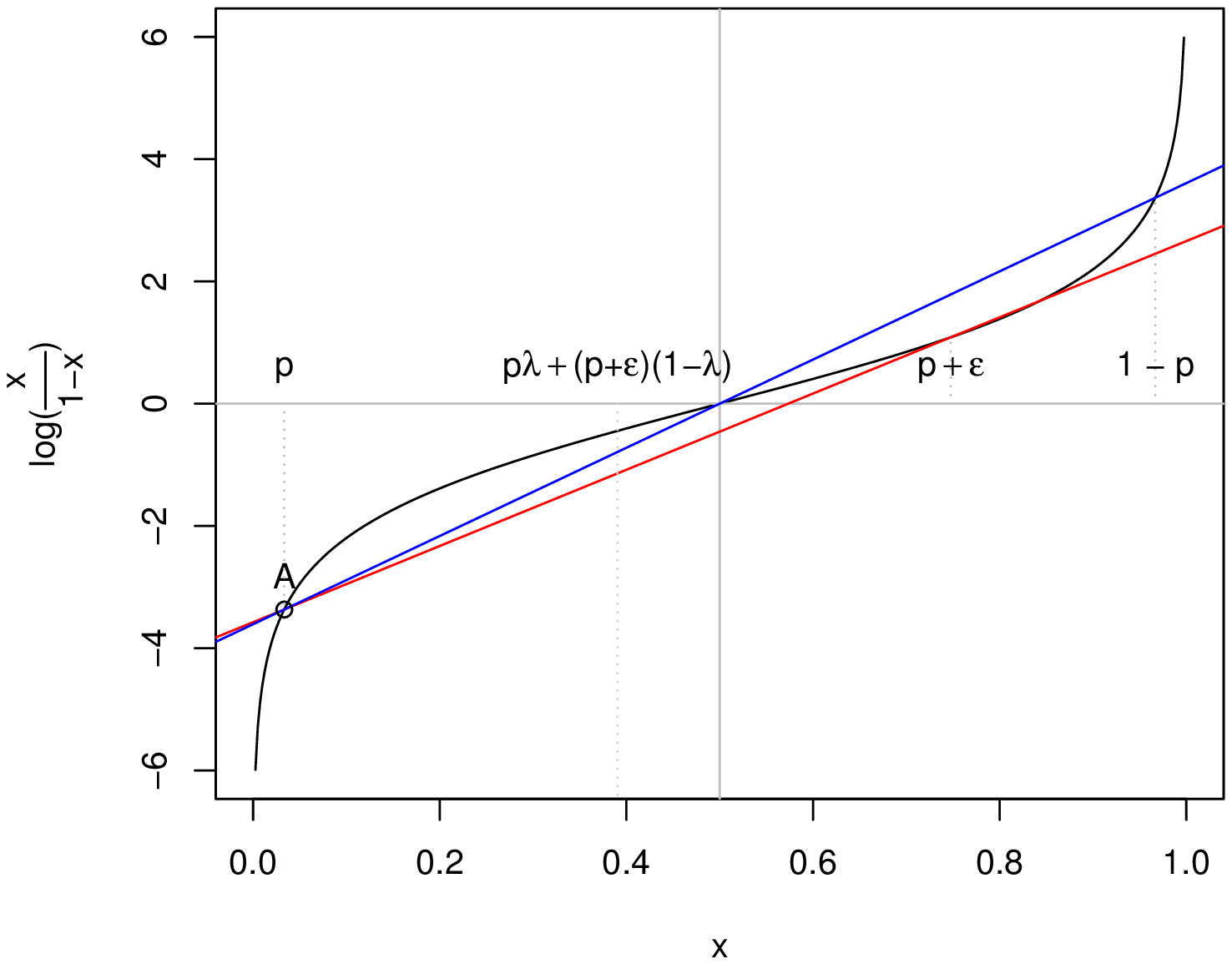}
   \caption{Phase transition: the two upper bounds cross at 0.25.}
   \label{fig:jensenpic}
 \end{figure}

For $p \in (\frac{1}{2} - \epsilon , \frac{1 - \epsilon}{2})$, we have the trivial inequalities
\begin{gather}
 \log(\frac{p}{p+\epsilon}) \leq \log(\frac{1-(p+\epsilon)}{1-p}), \label{eq:jensen2} \\
 \frac{p+\epsilon}{1-(p+\epsilon)} \leq \frac{1-p}{p}. \label{eq:jensen4}
\end{gather}
Equation \eqref{eq:jensen1}, among with Inequality \eqref{eq:jensen2} imply
\begin{align}
(\frac{p}{1-p})^{\lambda^*}(\frac{p+\epsilon}{1-(p+\epsilon)})^{1-\lambda^*} \geq 1. \label{eq:jensen3}
\end{align}
Inequality \eqref{eq:jensen3} combined with \eqref{eq:jensen4} reveals
\begin{align}
(\frac{p}{1-p})^{2\lambda^* - 1} \geq 1. \nonumber
\end{align}
Note that $\frac{p}{1-p} \leq 1$. Thus, we conclude $\lambda^* \leq \frac{1}{2}$. Hence, for every $\frac{1}{2} - \epsilon \leq p \leq \frac{1-\epsilon}{2}$,
\begin{align}
-\min_{0 \leq \lambda \leq \frac{1}{2}}  \log(f_{\lambda}(p,q)) = -\min_{0 \leq \lambda \leq 1}  \log(f_{\lambda}(p,q))  \label{eq:monotone3}
\end{align}
Combining inequalities \eqref{eq:monotone1} and \eqref{eq:monotone2}, with equality \eqref{eq:monotone3} reveals that $\mathcal{C}_{br}(p,q)$ is a decreasing function with respect to $p$ in the interval $[0,\frac{1-\epsilon}{2}]$. Hence, it takes its minimum value at $p = \frac{1 - \epsilon}{2}$, which completes the proof of Lemma \ref{lem:bernoullikomaki}.
\end{proof}

\begin{proof}[\textbf{Proof of Lemma \ref{lem:bernoulli}}]
Let $q - p = \delta \geq \epsilon$. Define
\begin{align}
f_{\lambda}(p,q) &=  p^{\lambda}q^{\lambda}  + p^{1-\lambda}q^{1-\lambda} \nonumber \\ 
& = p^{\lambda}(p+\delta)^{\lambda} + (1-p)^{1-\lambda}(1-(p+\delta))^{1-\lambda} = h(\delta).  \nonumber
\end{align}
We prove that the derivative of $h(\delta)$ with respect to $\delta$ is non-positive.
\begin{align*}
& \frac{\text{d}(h)}{\text{d}\delta} = \\
&  p^{\lambda}(1-\lambda)(p+\delta)^{-\lambda}  -  (1-p)^{\lambda}(1-(p+\delta))^{-\lambda}(1-\lambda) = \\
& (1-\lambda)(1-p)^{\lambda}(p+\delta)^{-\lambda} \! \left[ (\frac{p}{1-p})^{\lambda} - (\frac{p+\delta}{1-(p+\delta)})^{\lambda} \right] \stackrel{ \text{(c)}}{\leq} 0,
\end{align*}
where $(c)$ follows from the fact that $\frac{p}{1-p} \leq \frac{p+\delta}{1-(p+\delta)}$.
Hence, $h$ is a decreasing function of $\delta$. Thus, if we define $\tilde{q} = p + \epsilon$, we have
\begin{equation}
\mathcal{C}_{br}(p,q) = \min_{\lambda}f_{\lambda}(p,q) \geq \min_{\lambda} f_{\lambda}(p,\tilde{q}) = \mathcal{C}_{br}(p,\tilde{q}). \label{firststep}
\end{equation}
On the other hand, according to Lemma \ref{lem:bernoullikomaki}, 
\begin{equation}
\mathcal{C}_{br}(p,\tilde{q}) \geq \mathcal{C}_{br}(\frac{1-\epsilon}{2} , \frac{1+\epsilon}{2}). \label{secondstep}
\end{equation}
Furthermore, Bernoullies $(\frac{1-\epsilon}{2},\frac{1 + \epsilon}{2})$ and $(\frac{1+\epsilon}{2},\frac{1 - \epsilon}{2})$ satisfy the symmetry property required for applying Lemma \ref{lem:symmetry}. Hence, 
\begin{equation}
\mathcal{C}_{br}(\frac{1 - \epsilon}{2},\frac{1 + \epsilon}{2}) = -\log(\sqrt{1 - \epsilon^2}). \label{thirdstep}
\end{equation}
Inequalities \eqref{firststep} and \eqref{secondstep}, among with equality \eqref{thirdstep} complete the proof of Lemma \ref{lem:bernoulli}.
\end{proof}

\begin{proof}[\textbf{Proof of Lemma \ref{lem:example1}}]
Let $\ell_{d}$ be the column index which corresponds to the only unequal entry of vectors $v_{1}$ and $v_{2}$ ($\mathcal{H}(v_{1} , v_{2}) = 1$). According to Lemma \ref{lem:chernoffloss}, we can apply reduction to $\X^*_{1} , \X^*_{2}$ by removing all the columns either than $\ell_{d}$. Furthermore, by definition of $\X^*_{1} , \X^*_{2}$, for each column $\ell \neq \ell_{d}$, removing the $\ell$th columns from both of matrices does not incur any loss in $\textsf{CI}$. Hence we have 
\begin{equation}
\mathcal{C}(P_{\X^*_{1}} , P_{\X^*_{2}}) = \mathcal{C}(P_{\X^{\mathrm{br}}_{1}} , P_{\X^{\mathrm{br}}_{2}}), \label{first1}
\end{equation} 
where $X^{\mathrm{br}}_{1}$ and $X^{\mathrm{br}}_{2}$ correspond to column $\ell_{d}$ of $\X^*_{1}$ and $\X^*_{2}$. Moreover, $P_{\X^{\mathrm{br}}_{1}}$ and $P_{\X^{\mathrm{br}}_{2}}$ correspond to Bernoulli distributions with parameters $\frac{1 - \eta_{\N}}{2}$ and $\frac{1 + \eta_{\N}}{2}$ ($\eta_{\mathrm{N}} = \frac{1-2\mathrm{f}}{\mathrm{N}}$). Hence, 
\begin{equation}
\mathcal{C}(P_{\X^{\mathrm{br}}_{1}} , P_{\X^{\mathrm{br}}_{2}}) = \mathcal{C}_{\mathrm{br}}(\frac{1 - \eta_{\N}}{2} , \frac{1 + \eta_{\N}}{2}). \label{second1}
\end{equation}
On the other hand, Bernoullies $(\frac{1 - \eta_{\N}}{2} , \frac{1 + \eta_{\N}}{2})$ and $(\frac{1 + \eta_{\N}}{2} , \frac{1 - \eta_{\N}}{2})$ satisfy the symmetry condition of Lemma \ref{lem:symmetry}. Thus,
\begin{equation}
\mathcal{C}^{\mathrm{br}}(\frac{1 - \eta_{\N}}{2} , \frac{1 + \eta_{\N}}{2}) = -\log(\sqrt{1 - {\eta_{\N}}^2}). \label{third1}
\end{equation}
Combining equations \eqref{first1} , \eqref{second1} and \eqref{third1} completes the proof of Lemma \ref{lem:example1}.
\end{proof}
\begin{lem}[Almost closest pair for even $\N$]
Suppose $\mathrm{N} = 2n$. Consider two sequences $\upsilon_{1} , \upsilon_{2}$ with length $\mathrm{L}$, and Hamming distance one. In addition, suppose that $\upsilon_{0}$ has more ones that $\upsilon_{1}$. Define matrices $\mathcal{X}^*_{1}$ and $\mathcal{X}^*_{2}$, such that $\mathcal{X}^*_{1}$ has $n-1$ replicas of $\upsilon_{1}$ and $n+1$ replicas of $\upsilon_{2}$ as its rows, while $\mathcal{X}^*_{2}$ has $n$ replicas of $\upsilon_{1}$ and $n$ replicas of $\upsilon_{2}$. Then, for defined $\mathcal{X}^*_{1}$ and $\mathcal{X}^*_{2}$, 
$$\mathcal{C}(P_{\mathcal{X}^*_{1}} , P_{\mathcal{X}^*_{2}}) \leq -\log(\frac{\N-1}{\N}\sqrt{1 - {\eta_{\N - 1}}^2} + \frac{1}{\N}).$$ 
\end{lem}
\begin{proof}
Similar to the proof of Lemma \ref{lem:example1}, we can state that removing each column $\ell \neq \ell_{d}$ from matrices $\X^*_{1} , \X^*_{2}$ does not incur any loss in $\textsf{CI}$ ($\ell_{d}$ is defined as in proof of Lemma \ref{lem:example1}). Hence,
\begin{equation}
\mathcal{C}(P_{\X^*_{1}} , P_{\X^*_{2}}) = \mathcal{C}(P_{\X^{\mathrm{br}}_{1}} , P_{\X^{\mathrm{br}}_{2}}) =  \mathcal{C}_{\mathrm{br}}(\frac{1}{2} , \frac{1}{2} - \eta_{\N}). \nonumber
\end{equation} 
However,
\begin{align}
\mathcal{C}_{\mathrm{br}}(\frac{1}{2} ,  \frac{1}{2} - \eta_{\N}) = \nonumber \\ 
  -\log \Big(\min_{0 \leq \lambda \leq 1}  \Big[  ( \frac{1}{2}& - \frac{1-\f}{\N} + \frac{\f}{\N} ) ^\lambda ( \frac{1}{2} - \frac{\f}{\N} + \frac{\f}{\N} ) ^{1-\lambda} \nonumber + \\
  \!\!\!\!\  ( \frac{1}{2} - \frac{\f}{\N} & + \frac{1 - \f}{N} ) ^\lambda ( \frac{1}{2} - \frac{1-\f}{\N} + \frac{1-\f}{\N} ) ^{1-\lambda} \Big] \Big)  \stackrel{ \text{(d)}}{\leq} \nonumber  \\
   -\log \Big(\min_{0 \leq \lambda \leq 1}  \Big[  ( \frac{1}{2}& - \frac{1-\f}{\N} ) ^\lambda ( \frac{1}{2} - \frac{\f}{\N}) ^{1-\lambda} + \frac{\f}{\N} + \nonumber \\
   \!\!\!\!\  ( \frac{1}{2}  -  \frac{\f}{\N} & ) ^\lambda ( \frac{1}{2} - \frac{1-\f}{\N} ) ^{1-\lambda} + \frac{1-\f}{\N} \Big] \Big)= \nonumber  \\
  -\log \Big( \frac{\N - 1}{\N}  \min&_{ 0 \leq \lambda \leq 1}   \Big[  ( \frac{1 - \eta_{\N - 1}}{2} ) ^\lambda ( \frac{1 + \eta_{\N - 1}}{2} ) ^{1-\lambda} + \nonumber \\
  ( \frac{1 + \eta_{\N - 1}}{2} & ) ^\lambda ( \frac{1 - \eta_{\N - 1}}{2} ) ^{1-\lambda} \Big] + \frac{1}{\N} \Big) \stackrel{ \text{(e)}}{=}  \nonumber \\
  & \ \ \ \  \ \ \ \ \ -\log(\frac{\N-1}{\N}\sqrt{1 - {\eta_{\N - 1}}^2} + \frac{1}{\N}), \nonumber
\end{align}
where $(d)$ and $(e)$ follows from Lemmas \ref{lem:basic} and \ref{lem:symmetry} respectively. Hence,
\begin{equation}
\mathcal{C}^*(\N , \Ll) \leq -\log(\frac{\N-1}{\N}\sqrt{1 - {\eta_{\N - 1}}^2} + \frac{1}{\N}). \label{ineq:zoj}
\end{equation}
Inequality \eqref{ineq:zoj} together with lower bound $\tau_{1}$, reveal the bounds presented in the second part of Theorem \ref{thm:main}.
\end{proof}

\begin{proof}[\textbf{proof of Lemma \ref{lem:sufficientcondition}}]
First, we prove the initial part of the lemma. Assume that $(\X_{1} , \X_{2})$ is a critical pair of $\N \times \Ll$ matrices. We prove that $(\X_{1} , \X_{2})$ must have the expressed form, by induction on $\Ll$. According to Lemma \ref{lem:partitioning}, rows of $\X_{1}$ and $\X_{2}$ can be partitioned intro (i,j)-match quadruples. for desired $1 \leq i < j \leq \N$, we apply merging reduction by $\phi_{i,j}$ to obtain $(\X^{\prime}_{1} , \X^{\prime}_{2})$ with $\Ll - 1$ number of columns. Assume that for each $\Ll$, $\mathcal{U}^{(\Ll)}_{\mathrm{even}}$ and $\mathcal{U}^{(\Ll)}_{\mathrm{odd}}$ are sets consisting of sequences with length $\Ll$, having even number of ones and odd number of ones respectively. By hypothesis of induction, we know that there exist a number $n^*$, such that rows of $\delta(\X^{\prime}_{1})$ consist of $n^*$ replicas of sequences in $\mathcal{U}^{(\Ll - 1)}_{\mathrm{even}}$ (or $\mathcal{U}^{(\Ll-1)}_{\mathrm{odd}}$), while rows of $\delta(\X^{\prime}_{2})$ consist of $n^*$ replicas of sequences in $\mathcal{U}^{(\Ll - 1)}_{\mathrm{odd}}$ (or $\mathcal{U}^{(\Ll - 1)}_{\mathrm{even}}$). Because $(\X_{1} , \X_{2})$ is critical, Lemma \ref{lem:partitioning} states that there exist a disjoint partition of rows in $\delta(\X_{1}) , \delta(\X_{2})$ into (i,j)-match-quadruples, which corresponds to the new-(i,j)-match quadruples in the pair $(\delta(\X^{\prime}_{1}) , \delta(\X^{\prime}_{2}))$ (new-(i,j)-match-quadruples are defined and used in the proof of Lemma \ref{lem:chernoffloss2}). According to equation \ref{eq:clusters} with regards to the new-(i,j)-match quadruples, $n^*$ must be even. Now consider a new-(i,j)-match quadruple  $(s_{1} , r_{1} , s_{2} , r_{2})$ from rows of $\delta(\X^{\prime}_{1}) , \delta(\X^{\prime}_{2})$, which corresponds to the (i,j)-match quadruple with rows $s_{1} , r_{1}$ from $\delta(\X_{1})$ and $s_{2} , r_{2}$ from $\delta(\X_{2})$. In addition, assume that the new column produced by merging columns $i$ and $j$ is indexed as $\ell_{\mathrm{new}}$. 
For $u \in \{ 1,2 \}$, we define 
\begin{align}
\omega^u_{1} = \Big(\delta(\X_{u})(s_{u} , i) , \delta(\X_{u})(s_{u} , j) \Big), \\ 
\omega^u_{2} = \Big(\delta(\X_{u})(r_{u} , i) , \delta(\X_{u})(r_{u} , j) \Big).
\end{align}
For a desired matrix $\mathrm{X}$, we refer to its $n$th row as $\mathrm{X}(n)$. In addition, we refer to XOR function of the two entries in each of $\omega^u_{1}$ and $\omega^u_{2}$, by $\oplus(\omega^u_{1})$ and $\oplus(\omega^u_{2})$ respectively. From properties of match quadruples (as we mentioned in the proof of Theorem \eqref{lem:chernoffloss2}), and noting the definition of merging reduction, we have $\forall k \neq i , j$,
\begin{align}
 \delta(\X^{\prime}_{u})(s_{u} , k) = \delta(\X^{\prime}_{u})(r_{u} , k)
   = \delta(\X_{u})(s_{u} , k) = \delta(\X_{u})(r_{u} , k). \nonumber 
 \end{align}
 In addition,
 \begin{gather}
 \omega^u_{1} \neq \omega^u_{2}, \label{eq:unequality} \\
\delta(\X^{\prime}_{u})(s_{u} , \ell_{\mathrm{new}}) = \delta(\X^{\prime}_{u})(r_{u} , \ell_{\mathrm{new}}) = \oplus(\omega^u_{1}) = \oplus(\omega^u_{2}). \nonumber
\end{gather} 
Hence, the parity of sequences $\delta(\X^{\prime}_{u})(s_{u}) = \delta(\X^{\prime}_{u})(r_{u})$ are the same as $\delta(\X_{u})(s_{u})$ and $\delta(\X_{u})(r_{u})$. On the other hand, due to inequality \eqref{eq:unequality}, we obtain $\delta(\X_{u})(s_{u}) \neq \delta(\X_{u})(r_{u})$. 

We deduce that if the $\ell_{\mathrm{new}}$th entry in sequences $ \delta(\X^{\prime}_{u})(r_{u}) = \delta(\X^{\prime}_{u})(s_{u})$ is 0, then it corresponds to pairs $\omega^u_{1} , \omega^u_{2}$ in $\delta(\X_{u})(s_{u})$ and $\delta(\X_{v})(s_{v})$ respectively, where $\{\omega^u_{1} , \omega^u_{2} \} = \{ (0,0) , (1,1)\}$. In contrast, if the $\ell_{\mathrm{new}}$th entry in $\delta(\X^{\prime}_{u})(r_{u}) = \delta(\X^{\prime}_{u})(s_{u})$ is 1, then it corresponds to pairs $\omega^u_{1} , \omega^u_{2}$ in $\delta(\X_{u})(s_{u})$ and $\delta(\X_{u})(r_{u})$ respectively, where $\{\omega^u_{1} , \omega^u_{2} \} = \{ (0,1) , (1,0)\}$. Note that number of ones's parity does not change in both cases. Thus, we conclude that rows in $\delta(\X_{1})$ consist of $\frac{n^*}{2}$ replicas of sequences in $\mathcal{U}^{(\Ll)}_{\mathrm{even}}$, while rows in $\delta(\X_{2})$ consist of $\frac{n^*}{2}$ replicas of sequences in $\mathcal{U}^{(\Ll)}_{\mathrm{odd}}$, which completes the step of induction.

For the base of induction, where $\Ll = 1$, just note that the inequality condition on one dimensional vectors $\X_{1}$ and $\X_{2}$ means that $\delta(\X_{1})$ and $\delta(\X_{2})$ are not null. According to the definition of $\delta(\X_{1})$ and $\delta(\X_{2})$, one of them has all entries equal to one, while the other one has all entries equal to zero, which proves the base of induction.

Note that if $\X_{1} , \X_{2}$ have the expressed form, then removing each column clearly results in matrices to have equal multisets of rows. This means that $(\X_{1} , \X_{2})$ is cirical. Thus, we have proved the equivalency of the condition stated in the first part of Lemma \ref{lem:sufficientcondition}.

For the second part, assume that $(\X_{1} , \X_{2})$ satisfies the stated condition. For a given row $\mathbb{s}^{(1)}$ in $\X_{1}$, flip the $i$th entry to obtain sequence $\mathbb{s}^{(2)}$, flip the $j$th entry to obtain the sequence $\mathbb{s}^{(3)}$, and flip both of $i$th and $j$th entries to obtain sequence $\mathbb{s}^{(4)}$. By definition, Number of ones in $\mathbb{s}^{(4)}$ has the same parity as $\mathbb{s}^{(1)}$, while $\mathbb{s}^{(2)}$ and $\mathbb{s}^{(3)}$ have different parity in number of ones. Hence, $\mathbb{s}^{(4)}$ is a row of $\X_{1}$, while $\mathbb{s}^{(2)}$ and $\mathbb{s}^{(3)}$ are rows of $\X_{2}$. On the other hand, one can easily check that the defined quadruple $(\mathbb{s}^{(1)} ,  \mathbb{s}^{(4)} , \mathbb{s}^{(2)} , \mathbb{s}^{(3)})$ is actually an (i,j)-match-quadruple. Thus, according to the form of $\X_{1}$ and $\X_{2}$, we can partition their rows into (i,j)-match-quadruples. In addition, note that an (i,j)-match-quadruple satisfy the equality conditions $(1),(2)$ stated in Lemma \ref{lem:chernoffloss2}. Therefore, we conclude that applying merging reduction incur no loss in \textsf{CI} between the matrices. Furthermore, notice that by definition of merging reduction, the resulting matrices $\X^{\prime}_{1} , \X^{\prime}_{2}$ have the same form as $\X_{1} , \X_{2}$, such that rows in $\X^{\prime}_{1}$ consist of $2n_{1}$ replicas of each sequence in $\mathcal{U}^{(\Ll - 1)}_{even}$ and $2n_{2}$ replicas of each sequence in $\mathcal{U}^{(\Ll - 1)}_{odd}$, while  rows in $\X^{\prime}_{2}$ consist of $2n_{1}$ replicas of each sequence in $\mathcal{U}^{(\Ll - 1)}_{odd}$ and $2n_{2}$ replicas of each sequence in $\mathcal{U}^{(\Ll - 1)}_{even}$. Hence, we conclude that applying multiple merging reductions will not incur information loss in any reduction step, which completes the proof.  
\end{proof}

\begin{lem}[Near optimal pair in the noisy case] \label{lem:nearoptimal}
Define $\mathcal{L} , \mathrm{k} , \mathrm{R} , n , \epsilon_{\mathcal{L} , N}$ as in Theorem \ref{thm:main}. Then, there exist matrices $\X^*_{1} , \X^*_{2}$, such that 
\begin{align}
\mathcal{C}(P_{\X^*_{1}} , P_{\X^*_{2}}) \leq -\log(\sqrt{(\frac{\mathrm{N}-\mathrm{R}}{\mathrm{N}})^2 - \epsilon_{\mathcal{L},\mathrm{N}}^2} + \frac{\mathrm{R}}{\mathrm{N}}). \label{inq:upperbound}
\end{align} 
\end{lem}
For $\mathrm{R} = 0$, the upper bound in inequality \eqref{inq:upperbound} becomes equal to the lower bound $\tau_{2}$ we had on $\mathcal{C}$. Hence, \eqref{inq:upperbound} turns into equality. We conclude that Lemma \ref{lem:example2} is actually a special case of Lemma \ref{lem:nearoptimal}, where $R = 0$. 
\begin{proof}
Let $\mathbb{e}$ be the zero vector of length $\Ll$. First, we consider the case $\Ll \leq  \lfloor \log \mathrm{N} \rfloor + 1$, which yields $\mathcal{L} = \Ll$. 

Let $\X^*_{1}$ be the matrix whose rows consist of $n+1$ replicas of each sequence in $\mathcal{U}_{\mathrm{even}}$, $n$ replicas of each sequence in $\mathcal{U}_{\mathrm{odd}}$, and $\mathrm{R}$ replicas of $\mathbb{e}$. Similarly, let $\X^*_{2}$ be the matrix whose rows consist of $n$ replicas of each sequence in $\mathcal{U}_{\mathrm{even}}$, $n+1$ replicas of each sequence in $\mathcal{U}_{\mathrm{odd}}$, and $\mathrm{R}$ replicas of $\mathbb{e}$. In addition, for each sequence $\mathbb{s} \in \{ 0,1\}^{\Ll}$, define $n_{1}(\mathbb{s})$ to be the number of ones in $\mathbb{s}$. Moreover, let
\begin{align*}
& \mu(\mathbb{s}) = \begin{cases}1 &  \ n_{1}(\mathbb{s}) = \mathrm{even}\\0 &  \ n_{1}(\mathbb{s}) = \mathrm{odd}\end{cases}, \\
&\mathrm{E}_{\Ll} = \sum_{i = 1}^{\lfloor \N \rfloor} \left(\begin{array}{c}\Ll\\ 2i\end{array}\right) \f^{2i}(1-\f)^{\Ll - 2i}, \\
&\mathrm{O}_{\Ll} = \sum_{i = 1}^{\lceil \N \rceil - 1} \left(\begin{array}{c}\Ll\\ 2i+1\end{array}\right) \f^{2i+1}(1-\f)^{\Ll - (2i+1)}. 
\end{align*}
Then,
\begin{align}
P_{\X^*_{1}}(\mathbb{s}) = \frac{n + \mathrm{E}^{\mu(\mathbb{s})}_{\Ll}\mathrm{O}^{\overline{\mu(\mathbb{s})}}_{\Ll} + \mathrm{R}\f^{n_{1}(\mathbb{s})}(1-f)^{\Ll - n_{1}(\mathbb{s})}}{\N}, \nonumber \\
P_{\X^*_{2}}(\mathbb{s}) = \frac{n + \mathrm{E}^{\overline{\mu(\mathbb{s})}}_{\Ll}\mathrm{O}^{\mu(\mathbb{s})}_{\Ll} + \mathrm{R}\f^{n_{1}(\mathbb{s})}(1-f)^{\Ll - n_{1}(\mathbb{s})}}{\N}.
\end{align}
Note that $\mathrm{E}_{\Ll} - \mathrm{O}_{\Ll} = (1-2\f)^{\Ll}$. Hence, for $\zeta = (1 - 2\f)^\Ll$ we have
$$\mathrm{E}_{\Ll} = \frac{1 + \zeta}{2} , \mathrm{O}_{\Ll} = \frac{1 - \zeta}{2},$$
and,
\begin{align}
P_{\X^*_{1}}(\mathbb{s}) = \frac{n + (\frac{1+\zeta}{2})^{\mu(\mathbb{s})} (\frac{1-\zeta}{2})^{\overline{\mu(\mathbb{s})}}+ \mathrm{R}\f^{n_{1}(\mathbb{s})}(1-f)^{\Ll - n_{1}(\ell)}}{\N}, \nonumber \\
P_{\X^*_{2}}(\mathbb{s}) = \frac{n + (\frac{1+\zeta}{2})^{\overline{\mu(\mathbb{s})}}(\frac{1 - \zeta}{2})^{\mu(\mathbb{s})} + \mathrm{R}\f^{n_{1}(\mathbb{s})}(1-f)^{\Ll - n_{1}(\mathbb{s})}}{\N}. \nonumber
\end{align}
$P_{\X^*_{1}}(\mathbb{s})$ and $P_{\X^*_{2}}(\mathbb{s})$ only depend on the number of ones in $\mathbb{s}$. Therefore, by pairing each sequence which has odd number of ones, with a sequence having even number of ones, we observe that the symmetry condition of Lemma \ref{lem:symmetry} is satisfied. Hence, according to Lemma \ref{lem:symmetry},
\begin{align}
&\mathcal{C}(P_{\X^*_{1}} , P_{\X^*_{2}}) = - \! \min_{0 \leq \lambda \leq 1} \log\Big( \! \sum_{\mathbb{s} \in \{ 0 , 1 \} ^{\Ll}} P_{\X^*_{1}}(\mathbb{s})^{\lambda} P_{\X^*_{2}}(\mathbb{s})^{1 - \lambda}\Big) \! \stackrel{ \text{(f)}}{\leq} \nonumber \\ 
& -\min_{0 \leq \lambda \leq 1} \log \Big( \nonumber \\
&\!\!\! \!\!\  \sum_{\mathbb{s} \in \{ 0 , 1 \} ^{\Ll}} \!\! (\frac{n \! + \! (\frac{1 + \zeta}{2})^{\mu(\mathbb{s})}(\frac{1 - \zeta}{2})^{\overline{\mu(\mathbb{s})}}}{\N})^{\lambda}(\frac{n \! + \! (\frac{1 + \zeta}{2})^{\overline{\mu(\mathbb{s})}}(\frac{1 - \zeta}{2})^{\mu(\mathbb{s})}}{\N})^{1 - \lambda}  \nonumber \\
& + \sum_{\mathbb{s} \in \{ 0,1 \}^\Ll}\frac{\mathrm{R}\f^{n_{1}(\mathbb{s})}(1-f)^{\Ll - n_{1}(\mathbb{s})}}{\N} \Big)  = \label{eq:ghool} \\
& -\min_{0 \leq \lambda \leq 1} \log \Big( \nonumber \nonumber \\
&\!\!\! \!\!\  \sum_{\mathbb{s} \in \{ 0 , 1 \} ^{\Ll}} \!\! (\frac{n \! + \! (\frac{1 + \zeta}{2})^{\mu(\mathbb{s})}(\frac{1 - \zeta}{2})^{\overline{\mu(\mathbb{s})}}}{\N})^{\lambda}(\frac{n \! + \! (\frac{1 + \zeta}{2})^{\overline{\mu(\mathbb{s})}}(\frac{1 - \zeta}{2})^{\mu(\mathbb{s})}}{\N})^{1 - \lambda}  \nonumber  \\
& \ \ \ \  \ \ \ \ \ \ + \frac{\mathrm{R}}{\mathrm{N}} \Big),  \label{eq:ghool}  
\end{align}
where inequality $(\mathrm{f})$ follows from Lemma \ref{lem:basic}. Note that the terms $J_{1}(\mathbb{s}) = \frac{n \! + \! (\frac{1 + \zeta}{2})^{\mu(\mathbb{s})}(\frac{1 - \zeta}{2})^{\overline{\mu(\mathbb{s})}}}{\N}$ and $J_{2}(\mathbb{s}) = \frac{n \! + \! (\frac{1 + \zeta}{2})^{\overline{\mu(\mathbb{s})}}(\frac{1 - \zeta}{2})^{\mu(\mathbb{s})}}{\N}$ in \eqref{eq:ghool} only depend on the parity of number of ones in $\mathbb{s}$. Therefore, by pairing each sequence which has odd number of ones, with a sequence having even number of ones, we observe that the symmetry condition of Lemma \ref{lem:symmetry} is satisfied for distributions $J_{1} , J_{2}$. By Lemma \ref{lem:symmetry} we have
\begin{align}
\mathcal{C}(P_{\X^*_{1}} , & P_{\X^*_{2}})   \leq  -\log \Big( \sum_{\mathbb{s} \in \{ 0 , 1 \} ^{\Ll}} \!\! \sqrt{J_{1}(\mathbb{s})J_{2}(\mathbb{s}) } + \frac{\mathrm{R}}{\N}  \Big)  \nonumber \\
& = -\log(  2^{\Ll-1}\sqrt{\frac{2n + (1+\zeta)}{\N}. \frac{2n + (1-\zeta)}{\N}} \ + \ \frac{\mathrm{R}}{\N} )  \nonumber \\
& = -\log(  2^{\Ll-1}\sqrt{\frac{\mathrm{k} + \zeta}{\N}. \frac{\mathrm{k} -\zeta}{\N}} \ + \ \frac{\mathrm{R}}{\N} )  \nonumber \\
& = -\log( \sqrt{(\frac{2^{\Ll-1}.\mathrm{k}}{\N})^2 - \epsilon_{\mathcal{L} , \N}^2} \ + \ \frac{\mathrm{R}}{\N} )  \nonumber \\
&  = -\log( \sqrt{(\frac{\N-\mathrm{R}}{\N})^2 - \epsilon_{\mathcal{L} , \N}^2} \ + \ \frac{\mathrm{R}}{\N} ).\label{ineq:mainupperbound}
\end{align}

If $\Ll >  \lfloor \log \mathrm{N} \rfloor + 1$ was the case, then we define the first $\Ll - \mathcal{L}$ columns of both of $\X^*_{1} , \X^*_{2}$ to have all of their entries equal to zero. This way, according to Lemma \ref{lem:chernoffloss}, removing these columns does not incur any loss in \textsf{CI}. Thus, the above proof works for this case as well.
\end{proof}
Inequality \eqref{ineq:mainupperbound}, among with lower bound $\tau_{2}$, reveals the bounds in part 3 of Theorem \ref{thm:main}.

\begin{proof}[\textbf{Generalized Theorem}]
We don't bring the detailed proof for the generalized verison, as the sketch of proof is almost the same. 

If we name the flip probability of columns which are not removed during the reduction steps as $\{ \f_{\kappa_{i}} \}^{\mathrm{h}}_{i=1}$, then similar to inequality \eqref{eq:l1-distance}, we can drive
$$| p_{\mathrm{br},1} - p_{\mathrm{br},2} | \geq \frac{2^{\mathrm{h}-1}\prod_{i=1}^{\mathrm{h}}(1-2\f_{\kappa_{i}})}{\N},$$
which by the same procedures explained, leads to the lower bound for $\mathcal{C^*}(\N , \Ll , \mathrm{F})$ stated in Theorem \ref{thm:generalized}. In addition, using the necessary and sufficient condition obtained in Lemma \ref{lem:sufficientcondition}, and making use of the same tricks used in the proof of Lemma \ref{lem:nearoptimal}, we can define the pair $(\X^*_{1} , \X^*_{2})$ in such a way to obtain an upper bound on $\mathcal{C}^*(\X_{1} , \X_{2})$, which is very close to our lower bound, and consequently retrieve the near-tight upper bound mentioned in Theorem \ref{lem:nearoptimal} for $\mathcal{C}^*(\N , \Ll , \mathrm{F})$.
\end{proof}


\bibliographystyle{IEEEtran}
\bibliography{completeversion.bib}


\end{document}